\documentclass[reqno]{amsart}
\usepackage[left=2.7cm,right=2.7cm,top=3.5cm,bottom=3cm]{geometry}
\usepackage[english]{babel}
\usepackage{pifont}
\usepackage[latin1]{inputenc}
\usepackage[T1]{fontenc}
\usepackage{comment}
\usepackage{MnSymbol}
\usepackage{amsmath}
\usepackage{amsfonts}
\usepackage{amsthm}
\usepackage{amscd}
\usepackage{upgreek}
\usepackage[all]{xy}
\usepackage{graphicx}
\usepackage[usenames,dvipsnames]{color}
\usepackage{mathrsfs}
\usepackage{caption}
\usepackage{tikz-cd} 
\usepackage{braket}
\usepackage[
    colorlinks = true,
    linkcolor = black,
    urlcolor  = black,
    citecolor = blue,
    bookmarksopen = true,
    backref = page
]{hyperref}

\newcommand{\field}[1]{\mathbb{#1}}

\newcommand{\F}{\field{F}}
\newcommand{\C}{\field{C}}

\DeclareMathOperator{\End}{End}

\DeclareMathOperator{\Spec}{Spec}
\DeclareMathOperator{\Pic}{Pic}

\DeclareMathOperator{\vol}{vol}

\DeclareMathOperator{\disc}{disc}
\DeclareMathOperator{\PGL}{PGL}

\DeclareMathOperator{\GL}{GL}
\DeclareMathOperator{\calp}{\mathcal{P}}
\DeclareMathOperator{\nrd}{nrd}

\newtheorem{theorem}{Theorem}
\newtheorem{lemma}[theorem]{Lemma}
\newtheorem{proposition}[theorem]{Proposition}
\newtheorem{corollary}[theorem]{Corollary}
\newtheorem{remark}[theorem]{Remark}

\title{Duke for Drinfeld}
\author{Francesco Maria Saettone}
\address{Department of Mathematics, Weizmann Institute of Science, Israel}
\email{francesco.saettone@weizmann.ac.il}

\begin{document}
\begin{abstract}
We prove a function field analogue of Duke's equidistribution theorem for CM points,
in the setting of Drinfeld--Stuhler modular curves. Our results thus extend, to the
Drinfeld setting, both Duke's theorem on the modular curve and S.-W.\ Zhang's
equidistribution in the case of Shimura curves. Equidistribution is reduced via a
Weyl criterion to the decay of toric periods, which Waldspurger's formula expresses
through central values of automorphic $L$-functions, bounded in Lindel\"of-strength
form by the Riemann Hypothesis over function fields. We work at  arbitrary level structures and in every positive characteristic.
\end{abstract}

\maketitle

\setcounter{tocdepth}{2}
\tableofcontents

\section{Introduction}

\medskip
\noindent\textbf{Duke's theorem.}\quad 
An imaginary quadratic field $\mathbb{Q}\big(\sqrt{-d}\big)$, of discriminant
$-d<0$, carries finitely many ideal classes, and each of them, through the theory of
complex multiplication, deposits a point on the modular curve. As $d$ grows these
{\em special} points multiply and one may
ask where they choose to sit. Linnik, under an auxiliary congruence hypothesis on
$d$, was the first to prove that they equidistribute \cite{li}; Duke, building on Iwaniec's
bounds for Fourier coefficients of half-integral weight forms, removed every
hypothesis and established the result in full \cite{duke}: the Heegner points of
discriminant $-d$ become \emph{equidistributed}, with respect to the hyperbolic
measure, as $d\to\infty$.

The method has since become a template, in the form recorded by Michel and Venkatesh \cite[\S2.2]{michel-venkatesh2006}.  One tests equidistribution against an orthonormal basis of automorphic forms (the Weyl criterion), recognizes each Weyl sum as a \emph{toric period}, and invokes
Waldspurger's formula to convert the square of that period into a central value of an
$L$-function. Equidistribution is then equivalent to the decay of these central
values, and decay is precisely what a \emph{subconvexity} bound provides. This is the
route taken, in ever greater generality, by Shou-Wu Zhang 
and Einsiedler--Lindenstrauss--Michel--Venkatesh
\cite{swz,ELMV3}, among others. Subconvexity is the engine;
it is also, over a number field, the expensive part.
\\

Let $p$ be a prime ($p=2$ is allowed), $q$ a power of $p$, and consider the finite field $\mathbb F_q$. This paper asks Duke's question after replacing $\mathbb{Q}$ by the function field
$F=\mathbb{F}_q(C)$ of a smooth, projective, geometrically connected curve  $C$ over $\mathbb F_q$. The dictionary
between number fields and function fields is by now a venerable one, and it has a
recurring punchline: theorems that are hard over $\mathbb{Q}$ may sometimes be available over $F$, because the geometry of $C$ supplies tools that the
arithmetic of $\mathbb{Z}$ withholds. The most dramatic instance is the Riemann
Hypothesis, which over function fields is not a hope but a theorem of Weil for
curves, of Deligne in general \cite{deligne-weilII}, and, in the automorphic
incarnation we shall need, of Drinfeld and Lafforgue \cite{drinfeld-elliptic-modules,lafforgue}.

On the geometric side, the modular curve is replaced by a \emph{Drinfeld--Stuhler
modular curve}.  These are the moduli spaces of $D$-elliptic sheaves introduced by
Laumon--Rapoport--Stuhler \cite{laumon-rapoport-stuhler}, attached to a quaternion
algebra $D$ over $F$ that is split at the chosen place $\infty$. When $D$ is a
division algebra the resulting curve is proper, the function field avatar of a compact
Shimura curve; when $D\simeq M_2(F)$ one recovers the classical Drinfeld modular curve
\cite{gekeler-drinfeld-curves}, which has cusps. Both cases
carry CM points, indexed by the quadratic extensions $E/F$ that embed into $D$, and
both cases inherit a Galois action through global class field theory.

\medskip
\noindent\textbf{Equidistribution of Drinfeld--Stuhler CM packets.}\quad
We prove a function field Duke theorem in both Drinfeld--Stuhler regimes. Let $G:=D^\times/F^\times$ be the algebraic group attached to $D$ and denote by $F_\infty$ the completion of $F$ at $\infty$. The
measure-theoretic home of the statement is not a moduli stack but the adelic quotient
\[
        X_{U_\star}
        =
        G(F)\backslash G(\mathbb A_F)/U_\star
        \qquad
        \text{for}
        \qquad
    U_\star=T_\star(F_\infty)K_f,
\]
where $K_f\subset G(\mathbb A_{F,f})$ is an arbitrary fixed open compact subgroup and
$T_\star(F_\infty)\subset G(F_\infty)$ is the (compact)  torus fixing a
chosen point $z_\star\in\Omega(F_{\infty,\star})$ in the Drinfeld upper half plane $\Omega$, with $\star\in\{2,\mathrm{ram}\}$, i.e., the unramified ($\star=2$)
or a fixed ramified ($\star=\mathrm{ram}$) quadratic extension $F_{\infty,\star}$ of $F_\infty$. The quotient $X_{U_\star}$ thus remembers the
prescribed CM type at $\infty$ and carries the normalized probability
measure $\mu_{U_\star}$.

To an admissible CM datum $(E,R)$, i.e., an imaginary separable quadratic extension
$E/F$ with fixed local type
$E_\infty\simeq F_{\infty,\star}$\ , together with an order $R\subset E$
satisfying the fixed-level realizability condition, we attach a finite CM packet
in $X_{U_\star}$ and its normalized counting measure $\mu_{E,R,\star}$. Its
complexity is measured by the 
\[
        \mathfrak D_{E,R}
        =
        \mathfrak D_{E/F}\mathfrak f_R^2
\]
the product of the discriminant $\mathfrak D_{E/F}$ of the field $E$ and the square of
the conductor $\mathfrak f_R$ of the order $R$. Letting this discriminant grow is the function field counterpart of sending
$d\to\infty$.

\begin{theorem}[Theorems~\ref{thm:compact-adic-duke}
and~\ref{thm:split-adic-duke}]
\label{thm:intro-duke}
Let $D$ be a quaternion algebra over $F$, split at $\infty$, set
$G=D^\times/F^\times$, and fix a local CM type
$\star\in\{2,\mathrm{ram}\}$ and a finite level subgroup
$K_f\subset G(\mathbb A_{F,f})$ as above. Let $(E_n,R_n)$ be a sequence of admissible CM data
with
\[
        (E_n)_\infty\simeq F_{\infty,\star}
        \qquad
        \text{and}
        \qquad
        \deg\mathfrak D_{E_n,R_n}\to\infty
\]
satisfying the finite component-character condition described in
\S\ref{subsubsec:reduced-norm-component-characters}.
Then
$\mu_{E_n,R_n,\star}$ weakly-$*$ converges to $\mu_{U_\star}$ on $X_{U_\star}$.

This holds in both regimes:
\begin{enumerate}
\item if $D$ is a division algebra, then $X_{U_\star}$ is compact;
\item if $D\simeq M_2(F)$, then $G=\PGL_2$ and $X_{U_\star}$ has finite volume but
is noncompact. 
\end{enumerate}
\end{theorem}
Each point of the packet has a prescribed local nature, see Remark~\ref{r:cm-local}. At the distinguished place
$\infty$, the condition $(E_n)_\infty\simeq F_{\infty,\star}$ forces the CM point to
sit at the fixed point of the elliptic torus $T_\star(F_\infty)$ in the Drinfeld
upper half-plane: this is the local CM type, unramified ($\star=2$) or ramified
($\star=\mathrm{ram}$), that the quotient $X_{U_\star}$ is built to remember. At the finite places, the order in the datum and the level $K_f$ prescribe the
remaining local data, assembled into the finite adelic coordinate of the point. 

The finite component-character condition costs nothing once the field discriminant
itself grows, namely $\deg\mathfrak D_{E_n/F}\to\infty$,
since $\mathfrak X_{U_\star}$ is finite and depends only on the fixed level. It is a
real condition only in families where the field $E_n$ is fixed and the conductor of
$R_n$ alone tends to infinity.

When the finite level $K_f$ descends from a level subgroup
$K_I\subset D^\times(\mathbb A_{F,f})$, the adelic statement acquires a geometric
meaning through rigid-analytic uniformization, see also the diagram \eqref{diagram}.  We underline that, as explained in Remark~\ref{r:XU-not-Bun2}, this is not an equidistribution statement on $\operatorname{Bun}_2$, the moduli space of $\PGL_2$-bundles  over $\mathbb P^1_{\mathbb F_q}$, nor on $\operatorname{Bun}_2$ with level.

Consider $\Delta_I
        :=
        F^\times\backslash\mathbb A_F^\times/
        \bigl(F_\infty^\times(K_I\cap\mathbb A_{F,f}^{\times})\bigr)$.
Then $X_{U_\star}$ maps canonically to the central quotient
\[
        M_{\mathcal D,I}^{\operatorname{an}}(\mathbb C_\infty)/\Delta_I,
\]
where $\mathbb C_\infty$ is the completion of an algebraic closure of $F_\infty$, and there the theorem becomes equidistribution of the corresponding projectivized
analytic CM packets.  The full level-$I$ curve
\[
        M_{\mathcal D,I}^{\operatorname{an}}(\mathbb C_\infty)
\]
is reached by passing to the lifted quotient $\widetilde X_{I,\star}$.  To our
knowledge, the compact case is the first instance of Duke's theorem over a function
field in the quaternionic, compact Shimura-curve setting.

\medskip
\noindent\textbf{The shape of the proof.}\quad
The architecture is the one explained  in
\cite[\S2.2]{michel-venkatesh2006}. A Weyl criterion reduces equidistribution to the decay of Weyl sums; the packet
sums are identified with toric periods $P_\chi^D(f)$ and Qiu's function field
Waldspurger formula \cite{qiu} expresses the square of such a period as
\[
        |P_\chi^D(f)|^2
        =
        \frac{L(2,1_F)\,L\!\left(\tfrac12,\Pi_E\otimes\chi\right)}
             {2\,L(1,\Pi,\operatorname{Ad})}
        \prod_v\alpha_v^\natural(f_v,\overline{f}_v;\chi_v),
\]
Here $f$ is an automorphic test vector, $\chi$ is a character of the finite CM packet, $\Pi$ is the $\GL_2$-automorphic representation attached to the automorphic representation generated by $f$, $\Pi_E$ is its base change to $E$,
$1_F$ is the trivial character of $F^\times\backslash\mathbb A_F^\times$, $\operatorname{Ad}$ denotes the adjoint $L$-function, and $\alpha_v^\natural$ are the normalized local Waldspurger factors. This formula turns the analytic problem into one about central values $L(\tfrac12,\Pi_E\otimes\chi)$. We stress that our results are unconditional in every positive characteristic, including the
even case. This rests on Qiu's function field Waldspurger formula \cite{qiu}, valid in every positive characteristic, in contrast to the previous formula of Chuang--Wei \cite{chuang-wei}, which holds in odd characteristic only.

Here is where the two worlds part company. Over a number field, one now needs a
genuine subconvexity bound, and earning it is the hard labor of the subject.
Over $F$ the same bound is essentially free: the completed $L$-function is a
polynomial in $q^{-s}$ whose zeros, by the Riemann Hypothesis of Drinfeld and
Lafforgue, lie on the critical line, and the Grothendieck--Ogg--Shafarevich formula
\cite{sga5} controls its degree by the conductor. A textbook Littlewood argument
under RH \cite{iwaniec-kowalski} then yields a Lindel\"of-strength bound, which is
stronger than any subconvexity exponent one would have fought for. 

This is emphatically not to say the proof is free, as we deal with the following three issues.
\\\emph{First}, the local periods $\alpha_v^\natural$ must be controlled uniformly as
the datum $(E,R)$ varies, which requires a careful comparison of the valuation
filtrations on $D$ and on $E$, delicate precisely at the non-maximal orders and at
the places where $D$ ramifies, and valid even in residue characteristic $2$. 
\\\emph{Second}, the cuspidal estimate does not see the
one-dimensional automorphic spectrum; this finite remainder is governed by the
reduced norm and dispatched by an explicit analysis of component characters.
\\\emph{Third}, and only in the split case, one must contend with the continuous
spectrum: the toric period of an Eisenstein series is unfolded to an abelian Hecke
zeta integral, whose global aspect is again controlled by the function field RH and
whose local aspect rests on the uniform bounds of \cite{ELMV3}.

\medskip
\noindent\textbf{Organization.}\quad
Section~\ref{sec:drinfeld-modules} recalls the Drinfeldian setting: $D$-elliptic
sheaves and Drinfeld--Stuhler modules in the sense of
\cite{laumon-rapoport-stuhler}, the moduli curves and their rigid analytic
uniformization at $\infty$, the CM points and their adelic packets, and the
homogeneous quotient $X_U$ on which the limiting measures live.

Section~\ref{sec:3} carries out the equidistribution argument. We first assemble
the analytic input common to both cases: a Weyl criterion reduces equidistribution
to the decay of packet sums, these are identified with toric periods, and Qiu's
function field Waldspurger formula \cite{qiu} expresses the square of each period
through a central value $L(\tfrac12,\Pi_E\otimes\chi)$, after fixing the local
normalizations that control the periods $\alpha_v^\natural$ uniformly in the datum
$(E,R)$. We then gather the function field Lindel\"of bound, a consequence of
the Riemann Hypothesis of Drinfeld and Lafforgue
\cite{drinfeld-elliptic-modules,lafforgue} together with the conductor-degree
control of \cite{sga5} and deduce the resulting cuspidal decay. With this in
hand we prove the compact theorem, and finally the split theorem, where the
continuous spectrum is absent from the compact case and requires a separate analysis
of the Eisenstein contribution, resting globally again on the function field RH and
locally on the uniform bounds of \cite{ELMV3}.

\medskip
\noindent\textbf{Applications and future directions.}\quad
Equidistribution of CM points is rarely an end in itself; over number fields it has
become an important input to problems in Diophantine geometry. The
distribution of Heegner points and of singular moduli underlies, for instance,
Habegger's theorem that only finitely many singular moduli are algebraic units
\cite{habegger}, and Binyamini's effective results toward Andr\'e--Oort in
$Y(1)^n$ \cite{gal}. We expect the function field theorems proved here to play the
analogous role for Drinfeld and Drinfeld--Stuhler modules: a first application will be mentioned in \cite{bns}.

A more immediate direction is internal to the present circle of ideas. In
\cite{sae,sae2} we studied equidistribution of CM points on
special fibers  of Shimura curves: it would be natural to carry that work over to
the Drinfeld--Stuhler setting, as partially done by \cite{alvarado} and \cite{egmpz}, which are analogues of the equidistribution studied in \cite{michel}.

Taking a step further, we hope in a future work to investigate a Drinfeld analogue of the
\emph{mixing conjecture} of Michel and Venkatesh \cite[Conjecture~2]{michel-venkatesh2006}, which predicts the joint
equidistribution of pairs of CM points. Over number fields this question has
seen striking recent progress, in Khayutin's proof of joint equidistribution under a
splitting hypothesis \cite{kha} and in the resolution of the mixing conjecture under
GRH by Blomer, Brumley, and Khayutin \cite{brumley}. 
Over function fields, the first
analogue was studied by Shende and Tsimerman \cite{shende-tsimerman} for $\PGL_2$-bundles on $\mathbb P^1_{\F_p}$ (and subsequently established in \cite{sawin}).  A Drinfeld--Stuhler counterpart seems to us an especially natural extension: the moduli of $D$-elliptic sheaves  sits within the framework of shtukas for division algebras
\cite{laumon-rapoport-stuhler}, which generalize Drinfeld modules and provide the
function field analogue of certain quaternionic Shimura varieties \cite[\S3]{yun}.

\subsubsection*{Acknowledgments}
Some of the ideas from which this work stems were developed during a visit to
Institut de Math\'ematiques de Jussieu--Paris Rive Gauche: I thank Farrell Brumley
for his hospitality and for interesting discussions. The author was supported by the ERC, SharpOS, 101087910 and  ISF grants 2067/23, 1963/20.

\section{Drinfeldian setting}
\label{sec:drinfeld-modules}

Let us here recall some standard notation we will freely use throughout this work.

Let $C/\mathbb F_q$ be a smooth, projective, geometrically connected curve with
function field
$F=\mathbb F_q(C)$.
\\Fix a closed point $\infty\in C$, and put
\[
        A:=\Gamma(C-\{\infty\},\mathcal O_C).
\]
Thus $A$ is the Dedekind ring of functions regular away from $\infty$. 
A separable quadratic extension $E/F$ will be called {\em imaginary}, with respect to the fixed place $\infty$, if
\[
        E_\infty:=E\otimes_FF_\infty
\]
is a field.   
We write
$F_v$ for the completion of $F$ at a place $v$, $\mathcal O_v$ for its ring of
integers, and
\[
        \mathbb A_F,\qquad
        \mathbb A_{F,f},\qquad
        \widehat A:=\prod_{v\neq\infty}\mathcal O_v
\]
for the full adele ring, the  adeles away from $\infty$, and the profinite completion of $A$.

If $S$ is an $\mathbb F_q$-scheme, set
$C_S:=C\times_{\mathbb F_q}S$.
We denote by $\sigma_S\colon S\rightarrow S$
the absolute $q$-Frobenius, and by
\[
\tau:=\operatorname{id}_C\times\sigma_S\colon C_S\rightarrow C_S
\]
the Frobenius acting only on the second factor.

An $A$-field is a pair $(L,\gamma)$, where $L$ is a field and $\gamma\colon A\rightarrow L$
is an $\mathbb F_q$-algebra homomorphism.  We shall work in {\em generic
$A$-characteristic}, i.e., $\ker(\gamma)=0$.
The map $\gamma$ then defines a characteristic point $c\colon\Spec L\rightarrow \Spec A\subset C$,
and we write $\Gamma_c\subset C_L$
for its graph.

We shall also denote by $\mathbb{C}_\infty$  the completion of an algebraic closure of $F_\infty$.

\subsection{Drinfeld--Stuhler modules}

\subsubsection{$D$-elliptic sheaves}
\label{subsubsec:D-elliptic-sheaves}

In this subsection, let $D$ be a quaternion algebra over $F$, split
at $\infty$, and let $\mathcal D$
be a maximal $\mathcal O_C$-order in $D$ such that $\text{H}^0(C-\{\infty\},\mathcal D)=\mathcal O_D$.
For an $\mathbb F_q$-scheme $S$, set
\[
        \mathcal D_S:=\mathcal D\boxtimes\mathcal O_S
\]
where we recall that $\boxtimes$ denotes the usual external tensor product.

Following \cite{laumon-rapoport-stuhler}, a {\em $D$-elliptic sheaf} over $S$, with pole at $\infty$, is a pair
\[
        (\mathcal E,\psi)
\]
where $\psi:S\rightarrow C-\bigl(\{\infty\}\cup\operatorname{Ram}(D)\bigr)$
is the characteristic morphism, and
\[
        \mathcal E
        =
        \bigl((\mathcal E_i)_{i\in\mathbb Z},
              (j_i)_{i\in\mathbb Z},
              (t_i)_{i\in\mathbb Z}\bigr)
\]
is the following data.

Each $\mathcal E_i$ is a locally free $\mathcal O_{C_S}$-module of rank $4$,
equipped with a right $\mathcal D_S$-action compatible with the
$\mathcal O_C$-action.  The maps
\[
        j_i:\mathcal E_i\hookrightarrow\mathcal E_{i+1}
        \qquad
        \text{and}
        \qquad
        t_i:\sigma^*\mathcal E_i\hookrightarrow\mathcal E_{i+1}
\]
are injective $\mathcal D_S$-linear morphisms.  They are required to commute, in the
sense that
\[
        j_i\circ t_{i-1}
        =
        t_i\circ \sigma^*j_{i-1}
        \colon
        \sigma^*\mathcal E_{i-1}\rightarrow \mathcal E_{i+1}.
\]

The following conditions are imposed for every $i\in\mathbb Z$.

\begin{enumerate}
\item \emph{Periodicity.}  Set $\ell:=2\deg(\infty)$.
Then
\[
        \mathcal E_{i+\ell}=\mathcal E_i(\infty),
        \qquad
        \mathcal E_i(\infty):=
        \mathcal E_i\otimes_{\mathcal O_{C_S}}
        \mathcal O_{C_S}(\infty\times S),
\]
and the composite
\[
        j_{i+\ell-1}\circ\cdots\circ j_i:
        \mathcal E_i\rightarrow \mathcal E_{i+\ell}
        =
        \mathcal E_i(\infty)
\]
is the canonical inclusion.

\item \emph{Pole condition.}  The cokernel of $j_i$ is supported on
$\infty\times S$.  If $p_S:C_S\rightarrow S$ is the projection, then $p_{S,*}\operatorname{coker}(j_i)$
is locally free of rank $2$ over $S$.

\item \emph{Zero condition.}  Let
\[
        \Gamma_\psi\colon S\rightarrow C_S
\]
be the graph of $\psi$.  The cokernel of $t_i$ is supported on
$\Gamma_\psi(S)$; equivalently, there is a locally free $\mathcal O_S$-module
$\mathcal L_i$ of rank $2$ such that
\[
        \operatorname{coker}(t_i)\simeq(\Gamma_\psi)_*\mathcal L_i.
\]
\end{enumerate}

The morphism $\psi$ is also called the zero of the $D$-elliptic sheaf.
When no confusion is possible, we suppress $\psi$ and simply write
$\mathcal E$.

A morphism  $\mathcal E\rightarrow \mathcal F $
of $D$-elliptic sheaves over $S$ is
an integer $n\in\mathbb Z$, together with $\mathcal D_S$-linear morphisms
$ f_i\colon\mathcal E_i\rightarrow \mathcal F_{i+n}$, for  $i\in\mathbb Z$,
compatible with the $j$-maps and the $t$-maps.  Thus $f_{i+1}\circ j_i=j'_{i+n}\circ f_i$ and $f_{i+1}\circ t_i=t'_{i+n}\circ\sigma^*f_i$.
An isomorphism is a morphism admitting an inverse of the same kind.

If $S=\Spec R$, then the pole condition implies that
\[
        H^0\bigl((C-\{\infty\})\times S,\mathcal E_i\bigr)
\]
is independent of $i$.  The maps $t_i$ induce a Frobenius operator on this
module, and the right $\mathcal D$-action gives an action of
\[
        \mathcal O_D=\text{H}^0(C-\{\infty\},\mathcal D).
\]
Thus a $D$-elliptic sheaf gives a Drinfeld--Stuhler
$\mathcal O_D$-module.  Conversely, over generic $A$-characteristic,
Drinfeld--Stuhler $\mathcal O_D$-modules and $D$-elliptic sheaves are equivalent,
up to the harmless shift of the index $i$.

\subsubsection{Drinfeld--Stuhler modules}
\label{subsubsec:drinfeld-stuhler-modules}

Let $D$ be a central simple $F$-algebra of degree $2$, so that $\dim_FD=4$
and assume throughout this subsection that $D$ is split at $\infty$.  Fix a
maximal $A$-order
\[
        \mathcal O_D\subset D.
\]

For an $A$-field $L$, one has $\End_{\mathbb F_q}(\mathbb G_{a,L}^d)
\simeq M_d(L\{\tau\})$.
Let $\partial:M_d(L\{\tau\})\rightarrow M_d(L)$
be the constant-term map.

Following \cite[Def.~2.1]{papikian-drinfeld-stuhler}, a {\em Drinfeld--Stuhler $\mathcal O_D$-module} over $L$ is an embedding of rings
\[
        \varphi\colon \mathcal O_D\hookrightarrow M_d(L\{\tau\}),
        \qquad
        b\mapsto \varphi_b,
\]
satisfying the following two conditions.

First, for every $b\in\mathcal O_D\cap D^\times$, the endomorphism
\[
        \varphi_b\colon\mathbb G_{a,L}^2\rightarrow\mathbb G_{a,L}^2
\]
is an isogeny and its kernel is a finite group scheme over $L$ of order $\#(\mathcal O_D/\mathcal O_D b)$.
Second, for $a\in A$, viewed as the center of $\mathcal O_D$, one has $\partial(\varphi_a)=\gamma(a)I_d$.

A morphism of Drinfeld--Stuhler modules $\varphi\rightarrow\psi$ is an element
$u\in M_d(L\{\tau\})$
such that $u\varphi_b=\psi_bu$
for all $b\in\mathcal O_D$.
The endomorphism ring is therefore
\[
        \End_L(\varphi)
        =
        \{u\in M_d(L\{\tau\}):u\varphi_b=\varphi_bu
        \text{ for all }b\in\mathcal O_D\}.
\]
\begin{remark}
{\em
When $D=M_d(F)$ and $\mathcal O_D=M_d(A)$, Morita equivalence identifies
Drinfeld--Stuhler $\mathcal O_D$-modules with Drinfeld $A$-modules of rank $d$
\cite[Theorem~2.20 and Remark~2.21]{papikian-drinfeld-stuhler}.  When $D$ is
division, Drinfeld--Stuhler modules exist only over fields which split $D$
\cite[Lemma~2.5 and Remark~2.7]{papikian-drinfeld-stuhler}.
}
\end{remark}

\subsubsection{$D$-elliptic sheaves with CM}
\label{subsubsec:CM-drinfeld-stuhler}

Let $\mathcal E$ be a $D$-elliptic sheaf over
$\mathbb{C}_\infty$.  The central $A$-action on $\mathcal E$ gives a canonical embedding $A\hookrightarrow \End(\mathcal E)$.
We write
\[
        \End^0(\mathcal E)
        :=
        \End(\mathcal E)\otimes_A F
\]
for the rational endomorphism algebra.  Qiu's structure lemma shows that
$\End^0(\mathcal E)$ is a commutative field extension of $F$, of degree at most
$2$, admitting an $F$-embedding into $D$; moreover
\[
        \End^0(\mathcal E)\otimes_F F_\infty
\]
is a field, i.e. the extension is nonsplit at $\infty$
\cite[Lemma~2.2.1]{qiu}.

\subsection{Moduli spaces}
\label{subsec:moduli-and-adeles}

\subsubsection{Drinfeld--Stuhler  modular curves}
\label{subsubsec:drinfeld-stuhler-curves}

Recall that $D$ is a quaternion algebra over $F$, split at $\infty$.  Let
$\operatorname{Ram}(D)$ denote the finite set of places where $D$ is ramified.

Let $I\subset C-\{\infty\}$ be a nonempty finite closed subscheme.  
\begin{remark}{\em
Our work is in a generality that does not require to assume that $I$ is disjoint from $\operatorname{Ram}(D)$. }
\end{remark}
The level-$I$ moduli problem is considered over schemes $S$ whose characteristic morphism $\psi\colon S\rightarrow C-\big(\{\infty\}\cup\operatorname{Ram}(D)\big)$
has image disjoint from $I$.  

A level-$I$ structure
on a $D$-elliptic sheaf $\mathcal E$ is a trivialization
\[
        \mathcal D_I\boxtimes\mathcal O_S
        \xrightarrow{\ \sim\ }
        \mathcal E|_{I\times S},
        \qquad
        \mathcal D_I:=\mathcal D\otimes_{\mathcal O_C}\mathcal O_I,
\]
compatible with the $\mathcal D$-action and Frobenius.

Let
\[
        M_{\mathcal D,I}
\]
denote the corresponding moduli stack. By 
\cite[Theorems~4.1, 5.1 and Corollary~6.2]{laumon-rapoport-stuhler}, it is a smooth Deligne--Mumford stack of
relative dimension $1$ over
\[
        C-\bigl(I\cup\operatorname{Ram}(D)\cup\{\infty\}\bigr),
\]
and, for nontrivial level, is represented by a smooth quasi-projective curve.  If
$D$ is a division algebra, the curve is proper.  If $D\simeq M_2(F)$, one obtains
the usual Drinfeld modular curve, which is non-proper and has cusps (as extensively discussed in \cite{gekeler-drinfeld-curves}).

\subsubsection{Rigid analytic uniformization at \texorpdfstring{$\infty$}{infinity}}
\label{subsubsec:uniformization-infinity}

Let
\[
        \Omega
        :=
        \mathbb P^1(\mathbb{C}_\infty)-\mathbb P^1(F_\infty)
\]
be the Drinfeld upper half-plane.  After fixing an isomorphism
$D_\infty:=D\otimes_FF_\infty\simeq M_2(F_\infty)$,
the rigid analytic uniformization of the Drinfeld--Stuhler curve takes the form
\begin{equation}\label{e:analytic-unif}
        M_{\mathcal D,I}^{\operatorname{an}}(\mathbb{C}_\infty)
        \simeq
        D^\times(F)\backslash
        \left(
        \Omega\times D^\times(\mathbb A_{F,f})/K_I
        \right)
\end{equation}
where $K_I\subset D^\times(\mathbb A_{F,f})$ is the open compact subgroup determined
by the level structure (see \cite[Thm.~4.4.11]{blum-stuhler}).
Equivalently, after choosing representatives $d_1,\dots,d_h$ for
$D^\times\backslash D^\times(\mathbb A_{F,f})/K_I$,
one obtains a decomposition
\[
        M_{\mathcal D,I}^{\operatorname{an}}(\mathbb{C}_\infty)
        \simeq
        \bigsqcup_{i=1}^h
        \Gamma_i\backslash\Omega,
        \qquad
        \Gamma_i:=D^\times\cap d_iK_Id_i^{-1}.
\]
If $D$ is division, these quotients are cocompact.  If $D\simeq M_2(F)$, they are
the analytic components of the split Drinfeld modular curve, which are non-compact.

\subsubsection{CM points}
Let $E/F$ be an imaginary quadratic extension, i.e. a separable quadratic extension nonsplit at $\infty$. We say that
$\mathcal E$ has {\em CM} (complex multiplication) by $E$ if
\[
        \End(\mathcal E)\otimes_A F\simeq E .
\]
A point on a Drinfeld--Stuhler modular curve is called a CM point by $E$ if the
corresponding $D$-elliptic sheaf has CM by $E$
\cite[Definition~2.2.2]{qiu}.

The modular interpretation is compatible with the adelic action.  If $x$ is a CM
point represented by $\mathcal E$, then the action of $\bigl(\End(\mathcal E)\otimes_A F\bigr)^\times$
on the level data fixes $x$.  Hence, after fixing  an embedding $\rho\colon E\hookrightarrow D$,
the point $x$ lies in the fixed locus of $\rho(E^\times)$.  Conversely, a point
fixed by $\rho(E^\times)$ corresponds to a $D$-elliptic sheaf with CM by $E$
\cite[Lemmata~2.2.3--2.2.4]{qiu}.  Thus, on the full adelic tower, we obtain the ind-scheme
\[
        \operatorname{CM}
        =
        \bigcup_{\rho\colon E\hookrightarrow D}
        M(\mathbb{C}_\infty)^{\rho(E^\times)}
\]
see \cite[(2.6)]{qiu}.

Under the rigid analytic uniformization at $\infty$, this fixed-point description
becomes explicit.  Since $E_\infty/F_\infty$ is a field, the image of
$E_\infty^\times$ in $\PGL_2(F_\infty)$ is an elliptic torus and fixes a point
\[
        z_E\in\Omega(\mathbb{C}_\infty).
\]
After fixing the relevant embedding,
CM points are represented by adelic translates of this torus-fixed point:
\[
[\,g_\infty z_E,g_f\,]
\qquad 
\text{where}
\qquad
g=(g_\infty,g_f)\in D^\times(\mathbb A_F).
\]
in the notation of the adelic tower \cite[\S2.2.3]{qiu}.  At a fixed finite level,
one obtains the corresponding finite image on the quotient modular curve.

Finally,  CM points are algebraic: they are defined over
$E^{\mathrm{ab}}$, and the natural adelic action of $\mathbb A_E^\times/E^\times$
on the CM locus agrees with the Galois action through the reciprocity map of class
field theory \cite[Proposition~2.2.9 and Corollary~2.2.10]{qiu}.

\subsubsection{The adelic quotient and its measure}
\label{subsubsec:canonical-measure-CM-packets}

The passage to an adelic quotient is of the utmost necessity.  With the induced $\C_\infty$-topology, the space $\Omega$
is not locally compact, as $\mathbb{C}_\infty$ is not, and therefore it cannot carry the canonical Haar
measure theory needed for a Duke-type equidistribution statement.  We consequently replace the relevant locus by a locally compact adelic homogeneous quotient, on which quotient Haar measure is available.

  We keep the same notation
$D$ as in the preceding uniformization section.  Thus, in the compact
Drinfeld--Stuhler case, $D$ is the quaternion algebra defining the curve, while in
the split Drinfeld modular case we take
$D=M_2(F)$.
Consider the algebraic group
\[
        G:=D^\times/F^\times .
\]
After fixing an isomorphism $D_\infty\simeq M_2(F_\infty)$,
we identify $G(F_\infty)\simeq \PGL_2(F_\infty)$.
We also fix, once and for all, a quadratic field extension
\[
        F_{\infty,\star}/F_\infty,
        \qquad
        \star\in\{2,\mathrm{ram}\}
\]
where $F_{\infty,2}/F_\infty$ denotes the unramified quadratic extension, and
$F_{\infty,\mathrm{ram}}/F_\infty$ denotes a fixed ramified quadratic extension.
Choose
\[
        z_\star\in
        \Omega(F_{\infty,\star})
        =
        \mathbb P^1(F_{\infty,\star})-\mathbb P^1(F_\infty).
\]
Its stabilizer in $G(F_\infty)$ is the  maximal torus
\begin{equation}
\label{e:torus-star}        T_\star(F_\infty)
        :=
        \operatorname{Stab}_{G(F_\infty)}(z_\star)
        \simeq
        F_{\infty,\star}^{\times}/F_\infty^\times .
\end{equation}
In particular $T_\star(F_\infty)$ is compact.

Fix an open compact subgroup
$K_f\subset G(\mathbb A_{F,f})$
and set
\[
        U_\star:=T_\star(F_\infty)K_f
        \subset G(\mathbb A_F).
\]

Define
\begin{equation}\label{e:X_U}
        X_{U_\star}:=
        G(F)\backslash G(\mathbb A_F)/U_\star .
\end{equation}
This quotient has finite volume.  It is compact when $D$ is a division algebra, and
has finite volume but is non-compact when $D=M_2(F)$.

\begin{remark}
\label{r:XU-not-Bun2}
{\em
The quotient \eqref{e:X_U} should not be confused with the moduli stack $\operatorname{Bun}_2$, nor with
$\operatorname{Bun}_2$ equipped with finite level structure.
In the split case $D=M_2(F)$, the adelic description of $\operatorname{Bun}_2$ (as, for instance, in \cite[Appendix A.1]{shende-tsimerman})
involve quotienting by integral compact open
subgroups at every place; for $G=\PGL_2$, the local factor at $\infty$ would be
$\PGL_2(\mathcal O_\infty)$, as in the standard adelic description of
$\operatorname{Bun}_G$ (see, for instance, \cite[Appendix~A.1]{shende-tsimerman}).
In \eqref{e:X_U}, instead, the subgroup at $\infty$ is the compact (but not open in $G(F_\infty)$) torus
\eqref{e:torus-star}.
Thus $X_{U_\star}$ remembers the chosen CM local type at $\infty$ and is the
locally compact adelic quotient naturally attached to the Drinfeld upper-half-plane
uniformization. Indeed,  changing the finite level $K_f$ does not alter this
infinite place factor.
\\
We conclude this remark with the moduli-theoretic distinction: $\operatorname{Bun}_2$ parametrizes
vector bundles, whereas Drinfeld--Stuhler modular curves parametrize $D$-elliptic
sheaves, i.e., shtukas with a zero and a pole (which are the two {\em legs} in 
\cite[Example~2.3.3]{yun}). Level structure rigidifies the moduli problem but does not
remove these legs; it therefore leads to the Drinfeld--Stuhler curve, not to
$\operatorname{Bun}_2$.
}
\end{remark}

Let $dg$ be a Haar measure on $G(\mathbb A_F)$, and let
$\bar\mu_{U_\star}$ be the push-forward of the quotient measure on $G(F)\backslash G(\mathbb A_F)$
to $X_{U_\star}$.  We normalize it by
\[
        \mu_{U_\star}
        :=
        \frac{\bar\mu_{U_\star}}
             {\bar\mu_{U_\star}(X_{U_\star})}.
\]
Thus $\mu_{U_\star}$ is a probability measure, independent of the scalar
normalization of $dg$.

Assume now that $K_f$ is the image in $G(\mathbb A_{F,f})$ of the finite level
subgroup
\[
        K_I\subset D^\times(\mathbb A_{F,f})
\]
appearing in the rigid analytic uniformization \eqref{e:analytic-unif}.  Let
$\widetilde T_\star(F_\infty)\subset D_\infty^\times$
be the inverse image of $T_\star(F_\infty)$, and put
\[
        \widetilde X_{I,\star}
        :=
        D^\times(F)\backslash
        D^\times(\mathbb A_F)/
        \bigl(\widetilde T_\star(F_\infty)K_I\bigr).
\]
Then $z_\star$ gives the  map to the full level-$I$ analytic curve:
\[
        \widetilde\iota_{z_\star}\colon 
        \widetilde X_{I,\star}
        \rightarrow
        M_{\mathcal D,I}^{\operatorname{an}}(\mathbb{C}_\infty),
        \qquad
        D^\times(F)g(\widetilde T_\star K_I)
        \mapsto
        [g_\infty z_\star,g_f]_{K_I}.
\]

On the other hand, since $K_f$ is the image of $K_I$, its inverse image in
$D^\times(\mathbb A_{F,f})$ is
$\mathbb A_{F,f}^{\times}K_I$.
Hence
\[
        X_{U_\star}
        \simeq
        D^\times(F)\backslash
        D^\times(\mathbb A_F)/
        \bigl(\widetilde T_\star(F_\infty)\mathbb A_{F,f}^{\times}K_I\bigr).
\]
Thus $X_{U_\star}$ is the projectivized quotient of
$\widetilde X_{I,\star}$.

Let
\[
        \Delta_I
        :=
        F^\times\backslash\mathbb A_F^\times/
        \bigl(F_\infty^\times(K_I\cap\mathbb A_{F,f}^{\times})\bigr).
\]
It acts on $M_{\mathcal D,I}^{\operatorname{an}}(\mathbb{C}_\infty)$ by finite scalar
ideles $[z,g_f]_{K_I}\mapsto [z,a_fg_f]_{K_I}$.
The quotient $X_{U_\star}$ therefore gives a canonical map only after passing to
this central quotient:
\[
        \iota_{z_\star}^{\operatorname{proj}}\colon
        X_{U_\star}
        \rightarrow
        M_{\mathcal D,I}^{\operatorname{an}}(\mathbb{C}_\infty)/\Delta_I .
\]

The equidistribution theorem will be stated on $X_{U_\star}$ with respect to
$\mu_{U_\star}$.  Its geometric push-forward is canonically a statement on
$M_{\mathcal D,I}^{\operatorname{an}}(\mathbb{C}_\infty)/\Delta_I$.
For a statement on the full level-$I$ analytic curve itself, one works instead
with $\widetilde X_{I,\star}$ and $\widetilde\iota_{z_\star}$.

When $\star=2$, one shall consider
\[
        z_0:=z_2,\qquad
        T_0:=T_2,\qquad
        U:=U_2,\qquad
        X_U:=X_{U_2},\qquad
        \mu_U:=\mu_{U_2}.
\]
When $\star=\mathrm{ram}$, one shall instead use
\[
        z_{\mathrm{ram}},\qquad
        T_{\mathrm{ram}},\qquad
        U_{\mathrm{ram}},\qquad
        X_{U_{\mathrm{ram}}},\qquad
        \mu_{U_{\mathrm{ram}}}.
\]

\subsection{Adelic packets}

\subsubsection{Adelic CM packets}
\label{subsubsec:CM-packets}

Let $E/F$ be a separable quadratic extension satisfying
\[
        E_\infty:=E\otimes_FF_\infty\simeq F_{\infty,\star},
        \qquad
        \star\in\{2,\mathrm{ram}\}.
\]
Thus $E$ is nonsplit at $\infty$, and its local CM type is the fixed one chosen
in \S\ref{subsubsec:canonical-measure-CM-packets}.  We also impose the usual
embedding condition at $\operatorname{Ram}(D)$, that is, the set the finite ramified places of $D$:
\[
        E_v\ \text{is a field for every finite }v\in\operatorname{Ram}(D).
\]
Equivalently, $E$ embeds into $D$.  Fix the embedding $\iota\colon E\hookrightarrow D$, which induces an $F$-torus
\[
        T_E:=\operatorname{Res}_{E/F}\mathbb G_m/\mathbb G_m
        \subset
        G.
\]

Since $E_\infty\simeq F_{\infty,\star}$,
the compact torus
$T_E(F_\infty)=E_\infty^\times/F_\infty^\times$
is conjugate in $G(F_\infty)$ to the fixed torus
\[
        T_\star(F_\infty)
        =
        \operatorname{Stab}_{G(F_\infty)}(z_\star).
\]
Fix $z_\infty\in G(F_\infty)$
such that
$z_\infty^{-1}T_E(F_\infty)z_\infty=T_\star(F_\infty)$
and set
\[
        z:=(z_\infty,1_f)\in G(\mathbb A_F).
\]

Let $R\subset E$ be an $A$-order.  We say that $R$ is {\em realizable}, for the fixed
finite level datum, if there exists
$g_{R,f}\in G(\mathbb A_{F,f})$
such that, for any lift
$\widetilde g_{R,f}\in D^\times(\mathbb A_{F,f})$ of $g_{R,f}$, one has
\begin{equation}\label{e:realizable}
        \widehat R
        =\mathbb A_{E,f}\cap
        \widetilde g_{R,f}\widehat{\mathcal O}_D
        \widetilde g_{R,f}^{-1},
\end{equation}
inside $D(\mathbb A_{F,f})$, where $E$ is embedded in $D$ through the fixed
embedding $\iota:E\hookrightarrow D$.  
The condition is independent of the lift, since finite scalar adeles are central.

For such a choice, define the finite toric level subgroup
\[
        K_{T,R}
        :=
        T_E(\mathbb A_{F,f})
        \cap
        g_{R,f}K_fg_{R,f}^{-1}.
\]
If $K_f$ is contained in the image of
$\widehat{\mathcal O}_D^{\,\times}
        \subset D^\times(\mathbb A_{F,f})$
then $K_{T,R}
        \subseteq
        \widehat R^{\,\times}/\widehat A^{\,\times}$.
For maximal finite level equality holds.  At deeper level, $K_{T,R}$ is the
corresponding congruence subgroup of
$\widehat R^{\,\times}/\widehat A^{\,\times}$.

The projective (finite) CM parameter set is
\[
        \calp_{E,R}
        :=
        T_E(F)\backslash T_E(\mathbb A_{F,f})/K_{T,R}.
\]

Note that, since
$T_E(\mathbb A_{F,f})
\simeq\mathbb A_{E,f}^{\times}/\mathbb A_{F,f}^{\times}$,
the set $\calp_{E,R}$ is a finite level refinement of the projectivized ring class
group of $R$.  If $K_f$ is maximal at the finite places, so that $K_{T,R}=\widehat R^{\,\times}/\widehat A^{\,\times}$, then
\[
        \calp_{E,R}
        \simeq
        E^\times\backslash
        \mathbb A_{E,f}^{\times}
        /
        \bigl(\mathbb A_{F,f}^{\times}\widehat R^{\,\times}\bigr).
\]
For deeper level, $K_{T,R}$ is a congruence subgroup of
$\widehat R^{\,\times}/\widehat A^{\,\times}$, and $\calp_{E,R}$ is the
corresponding finite cover of the projectivized ring class group.
Define
\[
        g_{E,R}:=z(1,g_{R,f})\in G(\mathbb A_F).
\]
The CM packet is the image of
\[
x_{E,R,\star}\colon \calp_{E,R}\rightarrow X_{U_\star}
\qquad
        [t_f]
        \mapsto
        G(F)\,z(1,t_f)(1,g_{R,f})\,U_\star .
\]
This map is well defined.  Right multiplication by $K_{T,R}$ is absorbed by
$K_f$, while multiplication by $T_E(F)$ is absorbed on the left; at the 
place $\infty$ this uses
$z_\infty^{-1}T_E(F_\infty)z_\infty=T_\star(F_\infty)$.
\begin{remark}\label{r:cm-local}{\em
The points of $X_{U_\star}$ obtained in this way should be understood as
projectivized adelic avatars of CM points with fixed local type
$E_\infty\simeq F_{\infty,\star}$.  The choice of $z_\infty$ identifies the
$T_E(F_\infty)$-fixed point
\[
        z_E:=z_\infty\cdot z_\star\in\Omega(\mathbb C_\infty),
\]
where $z_\infty$ acts on $\Omega$ through $G(F_\infty)$.  This is the same
torus-fixed point $z_E$ appearing in the rigid-analytic description of CM points.
The finite adelic coordinate of $x_{E,R,\star}([t_f])$ is represented by
$t_f g_{R,f}\in G(\mathbb A_{F,f})$.
Thus $x_{E,R,\star}([t_f])$ records the corresponding CM point after projectivizing
by finite scalar ideles.

When $K_f$ is induced by a level subgroup
$K_I\subset D^\times(\mathbb A_{F,f})$, the precise relation with genuine analytic
CM points on the level-$I$ Drinfeld--Stuhler curve is described below using the
lifted parameter set $\widetilde{\calp}_{E,R,I}$ and the map
$\Theta_{E,R,I,\star}$.  
}
\end{remark}

We thus define the {\em CM packet} on the measurable quotient by
\[
        \operatorname{CM}_{E,R,\star}
        :=
        x_{E,R,\star}(\calp_{E,R})
        \subset X_{U_\star}\ .
\]
We attach to it the push-forward of the uniform probability measure on
$\calp_{E,R}$
\[
        \mu_{E,R,\star}
        :=
        (x_{E,R,\star})_*
        \left(
        \frac{1}{\#\calp_{E,R}}
        \sum_{[t_f]\in\calp_{E,R}}\delta_{[t_f]}
        \right).
\]
Thus, for any test function $\varphi$ on $X_{U_\star}$,
\[
        \int_{X_{U_\star}}\varphi\,d\mu_{E,R,\star}
        =
        \frac{1}{\#\calp_{E,R}}
        \sum_{[t_f]\in\calp_{E,R}}
        \varphi\!\left(x_{E,R,\star}(t_f)\right).
\]

The complexity of the admissible datum is measured as follows.  Denote by $\mathfrak f_R$
the conductor of $R$, and define the discriminant of the order $R$ by
\[
        \mathfrak D_{E,R}
        :=
        \mathfrak D_{E/F}\ \mathfrak f_R^2,
        \qquad
        D(E,R):=q^{\deg\mathfrak D_{E,R}}.
\]
This is the complexity parameter in  our equidistribution theorems.

A pair $(E,R)$ satisfying the conditions above, including the fixed local type
\[
        E_\infty\simeq F_{\infty,\star}
\]
and realizability at the fixed finite level $K_f$, will be called an {\em admissible CM datum}.  The auxiliary choices
\[
        \iota,\quad z_\infty,\quad g_{R,f},\quad \widetilde g_{R,f}
\]
shall be suppressed from the notation once fixed.

\subsubsection{Rigid analytic realization of CM packets}
\label{subsubsec:rigid analytic-CM-packets}

Suppose that the
finite level $K_f$ is the image in $G(\mathbb A_{F,f})$ of the subgroup
$K_I\subset D^\times(\mathbb A_{F,f})$
appearing in the rigid analytic uniformization \eqref{e:analytic-unif}.
Using the chosen lift $\widetilde g_{R,f}$, set
\[
        K_{E,R,I}
        :=
        \mathbb A_{E,f}^{\times}
        \cap
        \iota_f^{-1}
        \left(
        \widetilde g_{R,f}K_I\widetilde g_{R,f}^{-1}
        \right).
\]
Define the idelic CM parameter set
\[
        \widetilde{\calp}_{E,R,I}
        :=
        E^\times\backslash
        \mathbb A_{E,f}^{\times}/K_{E,R,I}.
\]
The CM map is
\[
        \Theta_{E,R,I,\star}:
        \widetilde{\calp}_{E,R,I}
        \rightarrow
        M_{\mathcal D,I}^{\operatorname{an}}(\mathbb{C}_\infty)
\qquad
        [t_f]
        \mapsto
        [\,z_E,\iota_f(t_f)\widetilde g_{R,f}\,]_{K_I}
        =
        [\,z_\infty z_\star,\iota_f(t_f)\widetilde g_{R,f}\,]_{K_I}.
\]

Indeed, also this map is well defined.  If
$u\in K_{E,R,I}$, then
$\iota_f(u)\widetilde g_{R,f}
\in\widetilde g_{R,f}K_I$,
so right multiplication by $u$ does not change the point of the level-$I$
quotient.  If $a\in E^\times$, then
$[\,z_E,\iota_f(at_f)\widetilde g_{R,f}\,]_{K_I}=[\,\iota_\infty(a)z_E,\iota_f(a)\iota_f(t_f)\widetilde g_{R,f}\,]_{K_I}=[\,z_E,\iota_f(t_f)\widetilde g_{R,f}\,]_{K_I}$,
because $E_\infty^\times$ fixes $z_E$, and the analytic uniformization is
quotiented on the left by $D^\times$.

Fix a lift $\widetilde z_\infty\in D^\times(F_\infty)$
of $z_\infty\in G(F_\infty)$.  Consider the lifted packet map
\[
        \widetilde x_{E,R,I,\star}\colon 
        \widetilde{\calp}_{E,R,I}
        \rightarrow
        \widetilde X_{I,\star}
\qquad
[t_f]\mapsto
        D^\times(F)
        \bigl(\widetilde z_\infty,\iota_f(t_f)\widetilde g_{R,f}\bigr)
        \bigl(\widetilde T_\star(F_\infty)K_I\bigr).
\]
Then the  compatibility with the full analytic curve is
\[
        \widetilde\iota_{z_\star}
        \bigl(\widetilde x_{E,R,I,\star}([t_f])\bigr)
        =
        \Theta_{E,R,I,\star}([t_f]).
\]

After projectivizing, let
\begin{equation}\label{e:quotient-map}
        q\colon
        M_{\mathcal D,I}^{\operatorname{an}}(\mathbb{C}_\infty)
        \rightarrow
        M_{\mathcal D,I}^{\operatorname{an}}(\mathbb{C}_\infty)/\Delta_I
\end{equation}
be the quotient map.  Then
$\iota_{z_\star}^{\operatorname{proj}}
        \bigl(x_{E,R,\star}(\mathrm{pr}([t_f]))\bigr)
        =
        q\bigl(\Theta_{E,R,I,\star}([t_f])\bigr)$.
Set
\[
        \operatorname{CM}_{E,R,I,\star}^{\operatorname{an}}
        :=
        \Theta_{E,R,I,\star}
        \bigl(\widetilde{\calp}_{E,R,I}\bigr)
        \subset
        M_{\mathcal D,I}^{\operatorname{an}}(\mathbb{C}_\infty).
\]
The following diagram summarizes the relation between the adelic CM packet and its
geometric realization in the rigid analytic Drinfeld--Stuhler curve:
\begin{equation}\label{diagram}
\begin{tikzcd}[column sep=large, row sep=large]
        \widetilde X_{I,\star}
        \arrow[r, "\widetilde\iota_{z_\star}"]
&
        M_{\mathcal D,I}^{\operatorname{an}}(\mathbb{C}_\infty)
\\
        \widetilde{\calp}_{E,R,I}
        \arrow[u, "\widetilde x_{E,R,I,\star}"]
        \arrow[r, "\Theta_{E,R,I,\star}"']
&
        \operatorname{CM}_{E,R,I,\star}^{\operatorname{an}}
        \arrow[u, hook]\;\;\;.
\end{tikzcd}
\end{equation}

\section{Equidistribution \`a la Duke}\label{sec:3}

\subsection{Weyl's sums and Waldspurger formula}
\label{subsec:common-analytic-input}

We collect the analytic ingredients used in both the compact and the split
non-compact cases.  
\\

Throughout this section, $D$ is either a quaternion division algebra over $F$,
split at $\infty$, or $D\simeq M_2(F)$.  Put
\[
        G:=D^\times/F^\times,
        \qquad
        [G]:=G(F)\backslash G(\mathbb A_F).
\]

Fix once and for all a local CM type
\[
        \star\in\{2,\mathrm{ram}\}.
\]
For the rest of Section~\ref{sec:3}, we use the shorthand
\[
        U:=U_\star=T_\star(F_\infty)K_f,
        \qquad
        X_U:=X_{U_\star}=[G]/U,
        \qquad
        \mu_U:=\mu_{U_\star}.
\]
Thus all CM data considered in this section satisfy
\[
        E_\infty\simeq F_{\infty,\star}.
\]
We also suppress $\star$ from the packet notation, writing
\[
        \mu_{E,R}:=\mu_{E,R,\star},
        \qquad
        x_{E,R}:=x_{E,R,\star}.
\]

Throughout Section~\ref{sec:3}, we assume that $K_f$ is the image in
$G(\mathbb A_{F,f})$ of an open compact subgroup
\[
        \mathcal K_f\subset \widehat{\mathcal O}_D^{\,\times}.
\]
Implicit constants attached to the fixed level may depend on $F$, $D$, $U$, and on
the chosen lift $\mathcal K_f$, but are uniform in the varying admissible CM datum
$(E,R)$ and in packet characters, unless further dependence is explicitly indicated.

\subsubsection{Weyl sums}
\label{subsubsec:weyl-sums}

Let $\varphi$ be a test function on $X_U$.  The Weyl sum of $\varphi$ over the
CM packet attached to $(E,R)$ is
\[
        W(\varphi;E,R)
        :=
        \int_{X_U}\varphi\,d\mu_{E,R}.
\]
Equivalently,
\[
        W(\varphi;E,R)
        =
        \frac{1}{\#\calp_{E,R}}
        \sum_{[t_f]\in\calp_{E,R}}
        \varphi\!\left(x_{E,R}(t_f)\right),
\]
where
\[
        \calp_{E,R}
        =
        T_E(F)\backslash T_E(\mathbb A_{F,f})/K_{T,R}.
\]

More generally, for a character
$\chi\colon\calp_{E,R}\rightarrow \mathbb C^\times$,
we consider the twisted Weyl sum
\[
        W(\varphi;E,R,\chi)
        :=
        \frac{1}{\#\calp_{E,R}}
        \sum_{[t_f]\in\calp_{E,R}}
        \varphi\!\left(x_{E,R}(t_f)\right)\chi([t_f])
\]
which indeed recovers the usual Weyl sum for $\chi=1$.

\subsubsection{Waldspurger formula}
\label{subsubsec:waldspurger}

Let $E/F$ be a separable quadratic extension embedded
in $D$, and let
\[
        \eta_E\colon F^\times\backslash\mathbb A_F^\times\rightarrow \{\pm1\}
\]
be the associated quadratic character.

Let $\Pi$ be a cuspidal automorphic representation of $\GL_2(\mathbb A_F)$ with
unitary central character.  Assume that $\Pi$ transfers to $D^\times(\mathbb A_F)$;
when $D=M_2(F)$, this means simply that we take $D^\times=\GL_2$.  We write
\[
        \pi=\Pi^D
\]
for the corresponding automorphic representation of $D^\times(\mathbb A_F)$.

Let $\chi:E^\times\backslash\mathbb A_E^\times\rightarrow \mathbb C^\times$
be a unitary Hecke character satisfying $\omega_\Pi\cdot \chi|_{\mathbb A_F^\times}=1$.
For $f\in\pi$, define the toric period
\[
        P_\chi^D(f)
        :=
        \int_{E^\times\mathbb A_F^\times\backslash\mathbb A_E^\times}
        f(\iota(t))\chi(t)\,d^\times t .
\]
We use the notation and normalizations as in \cite{qiu} for all $L$-functions.  Thus $1_F$ denotes
the trivial character of $F^\times\backslash\mathbb A_F^\times$, and $L(s,1_F)$ is
its global $L$-function.  We write $L(s,\Pi,\operatorname{Ad})$ for the adjoint
$L$-function of $\Pi$.  Finally, $\Pi_E$ denotes the quadratic base change of $\Pi$
to $\GL_2(\mathbb A_E)$, and $L(s,\Pi_E\otimes\chi)$ is the standard $L$-function of
$\Pi_E$ twisted by the Hecke character $\chi$.
Recall also that
\[
        \pi=\bigotimes_v'\pi_v
\]
is the restricted tensor product of its local components.  For pure tensors
\[
        f=\otimes_v' f_v,
        \qquad
        \widetilde f=\otimes_v'\widetilde f_v,
        \qquad
        \text{with}
        \quad
        f_v,\widetilde f_v\in\pi_v,
\]
Qiu's function field Waldspurger formula
\cite[Theorem~1.2.3]{qiu} gives

\begin{equation}
\label{eq:waldspurger-qiu}
        P_\chi^D(f)\,P_{\chi^{-1}}^D(\widetilde f)
        =
        \frac{
        L(2,1_F)\,
        L\!\left(\frac12,\Pi_E\otimes\chi\right)
        }{
        2\,L(1,\Pi,\operatorname{Ad})
        }
        \prod_v
        \alpha_v^\natural(f_v,\widetilde f_v;\chi_v).
\end{equation}
Here $\Pi_E$ denotes quadratic base change to $E$.  The local factor
$\alpha_v^\natural$ is Qiu's normalized local period
$\alpha^\sharp_{\pi_v}$ from \cite[\S1.2.3, local periods and (1.3)]{qiu}.  Namely,
if $(\cdot,\cdot)_v$ is the local pairing on $\pi_v\otimes\widetilde\pi_v$, then
\[
        \alpha_{\pi_v}(f_v,\widetilde f_v;\chi_v)
        :=
        \int_{E_v^\times/F_v^\times}
        (\pi_v(t)f_v,\widetilde f_v)_v\,\chi_v(t)\,dt,
\]
and
\[
        \alpha_v^\natural(f_v,\widetilde f_v;\chi_v)
        :=
        \frac{
        L(1,\eta_{E,v})L(1,\pi_v,\operatorname{Ad})
        }{
        L(2,1_{F_v})L(1/2,\Pi_{E,v}\otimes\chi_v)
        }\,
        \alpha_{\pi_v}(f_v,\widetilde f_v;\chi_v).
\]
With the standard unramified normalizations, this factor is equal to $1$ for all but
finitely many $v$.

In the unitary case, taking $\widetilde f=\overline f$ gives
\begin{equation}
\label{eq:waldspurger-square}
        |P_\chi^D(f)|^2
        =
        \frac{
        L(2,1_F)\,
        L\!\left(\frac12,\Pi_E\otimes\chi\right)
        }{
        2\,L(1,\Pi,\operatorname{Ad})
        }
        \prod_v
        \alpha_v^\natural(f_v,\overline f_v;\chi_v).
\end{equation}

For automorphic forms descending to $G=D^\times/F^\times$, the central character is
trivial.  Packet characters, being characters of the projective torus $T_E$, are
trivial on $\mathbb A_F^\times$.  Thus the compatibility condition
$\omega_\Pi\chi|_{\mathbb A_F^\times}=1$ is automatic in our application.

\subsubsection{Packet sums as toric periods}
\label{subsubsec:packet-sums-as-toric-periods}

We now relate the finite Weyl sums to the periods in
\eqref{eq:waldspurger-qiu}.  Keep the notation of
\S\ref{subsubsec:canonical-measure-CM-packets}.  Thus
\[
        z=(z_\infty,1_f),
        \qquad
        z_\infty^{-1}T_E(F_\infty)z_\infty=T_\star(F_\infty),
\qquad x_{E,R}(t_f)
        =
        G(F)\,z(1,t_f)(1,g_{R,f})\,U .
\]
Set $g_{E,R}:=z(1,g_{R,f})\in G(\mathbb A_F)$.
Let $f$ be a right $U$-invariant automorphic function on $[G]$, and let
$\varphi$ be the corresponding function on $X_U$.  Define the right translate
\[
        f_{E,R}(g):=f(g\,g_{E,R}).
\]
For $t=(t_\infty,t_f)\in T_E(\mathbb A_F)$, the choice of $z_\infty$ and the
right $T_\star(F_\infty)$-invariance of $f$ give
\[
        f_{E,R}(t)
        =
        f(t\,g_{E,R})
        =
        \varphi\!\left(x_{E,R}(t_f)\right).
\]
Consider the natural projection $\operatorname{pr}_f:\mathbb A_{E,f}^{\times}
        \rightarrow
        T_E(\mathbb A_{F,f})
        =
        \mathbb A_{E,f}^{\times}/\mathbb A_{F,f}^{\times}$
and set $\widetilde K_{T,R}:=
\operatorname{pr}_f^{-1}(K_{T,R})\subset\mathbb A_{E,f}^{\times}$.
Inflate a character
\[
        \chi:\calp_{E,R}\to\mathbb C^\times
\]
to a Hecke character of $E^\times\backslash\mathbb A_E^\times$, trivial on
$\mathbb A_F^\times$, on $E_\infty^\times$, and on
$\widetilde K_{T,R}$.  In other words, we regard $\chi$ as a character of
$T_E(F)\backslash T_E(\mathbb A_F)$
which is trivial on $T_E(F_\infty)$ and on $K_{T,R}$.

Let
\[
        K_T
        :=
        T_E(\mathbb A_F)\cap g_{E,R}Ug_{E,R}^{-1}
\]
so that  $K_T=T_E(F_\infty)K_{T,R}$.
Consider
\[
        w_{E,R}:=\#(T_E(F)\cap K_T).
\]
\begin{lemma}
With the above notation, one has $w_{E,R}\le q+1$.
\end{lemma}
\begin{proof}
Let $[\alpha]\in T_E(F)\cap K_T$, and choose a representative
$\alpha\in E^\times$.  Since $[\alpha]$ lies in a compact subgroup at every
place, the divisor of $\alpha$ on $E$ is Galois-invariant.  Hence $\sigma(\alpha)/\alpha\in k_E^\times$,
where $k_E$ is the constant field of $E$.  The map
$[\alpha]\mapsto \sigma(\alpha)/\alpha$
is injective, because $\sigma(\alpha)=\alpha$ implies $\alpha\in F^\times$.
Its image lies in
\[
        \{c\in k_E^\times:c\sigma(c)=1\},
\]
whose order is at most $q+1$.  Therefore $w_{E,R}\le q+1$.
\end{proof}

Since $K_T=T_E(F_\infty)K_{T,R}$,
the projection to the finite adeles gives a natural identification
\[
        T_E(F)\backslash T_E(\mathbb A_F)/K_T
        \simeq
        T_E(F)\backslash T_E(\mathbb A_{F,f})/K_{T,R}
        =
        \calp_{E,R}.
\]
The measure of each right $K_T$-orbit in
$T_E(F)\backslash T_E(\mathbb A_F)$ is
\[
        \frac{\vol(K_T)}{w_{E,R}}
\]
since the stabilizer of such an orbit is $T_E(F)\cap K_T$.
The toric period of $f_{E,R}$ is therefore
\[
        P_\chi^D(f_{E,R})
        =
        \int_{T_E(F)\backslash T_E(\mathbb A_F)}
        f_{E,R}(t)\chi(t)\,d^\times t,
\]
and the integrand is right $K_T$-invariant.  Hence

\begin{equation}
\label{eq:compare-weyl-period}
        P_\chi^D(f_{E,R})
        =\frac{\vol(K_T)}{w_{E,R}}
        \sum_{[t_f]\in\calp_{E,R}}
        \varphi\!\left(x_{E,R}(t_f)\right)\chi([t_f])=
        \frac{\vol(K_T)}{w_{E,R}}\,
        \#\calp_{E,R}\,
        W(\varphi;E,R,\chi).
\end{equation}

Combining \eqref{eq:compare-weyl-period} with
\eqref{eq:waldspurger-square} gives
\begin{equation}
\label{eq:weyl-waldspurger}
                |W(\varphi;E,R,\chi)|^2
        =
        \frac{w_{E,R}^{\,2}}
        {\vol(K_T)^2\,\#\calp_{E,R}^{\,2}}
        \cdot
        \frac{
        L(2,1_F)\,
        L\!\left(\frac12,\Pi_E\otimes\chi\right)
        }{
        2\,L(1,\Pi,\operatorname{Ad})
        }
        \prod_v
        \alpha_v^\natural
        \bigl(f_{E,R,v},\overline f_{E,R,v};\chi_v\bigr).
\end{equation}

For our equidistribution problem we shall thus deal with two estimates: a bound for the central value and a
local normalization estimate for the remaining product of local factors.

\subsubsection{Local packet normalization}
\label{subsubsec:local-normalization}

We now state the local input needed for the Waldspurger part: this is the
cuspidal analogue of the local harmonic analysis of
\cite[\S9]{ELMV3}.  It is a fixed-level estimate for toric matrix-coefficient
integrals. 

For a finite place $v$, put
\[
        a_v(E):=\operatorname{ord}_v(\mathfrak D_{E/F}),
        \qquad
        b_v(E,R):=\operatorname{ord}_v(\mathfrak D_{E,R}/\mathfrak D_{E/F}),
\]
so that
\[
        n_v(E,R)
        :=
        \operatorname{ord}_v(\mathfrak D_{E,R})
        =
        a_v(E)+b_v(E,R).
\]

Define
\[
        D_E:=q^{\deg\mathfrak D_{E/F}},
        \qquad
        \Delta_R:=\frac{D(E,R)}{D_E}
        =
        \prod_{v\neq \infty}q_v^{b_v(E,R)}.
\]

All discriminants here are discriminants of trace-dual lattices.  In particular,
the notation is valid without any tameness assumption and includes characteristic
$2$, provided $E/F$ is separable.

All estimates in this subsection are finite-place estimates.  Write
\[
        K_{T,R}=\prod_{v\neq\infty}K_{T,R,v}.
\]
Fix a finite place $v$, consider $F_v$, and choose a lift
$\widetilde g_{R,v}\in D_v^\times$of $g_{R,v}\in G(F_v)$.  Let
\[
        \mathcal O'_{D,v}
        :=
        \widetilde g_{R,v}\mathcal O_{D,v}\widetilde g_{R,v}^{-1},
        \qquad
        R_v:=E_v\cap\mathcal O'_{D,v}.
\]
Choose an open compact lift
$\mathcal K_{f,v}\subset\mathcal O_{D,v}^{\times}$
whose image in $D_v^\times/F_v^\times$ is $K_{f,v}$, and put
\[
        \mathcal K'_{f,v}
        :=
        \widetilde g_{R,v}\mathcal K_{f,v}\widetilde g_{R,v}^{-1}.
\]
Then $K_{T,R,v}$ is the image in $E_v^\times/F_v^\times$ of
$\widetilde K_{T,R,v}
        :=
        E_v^\times\cap F_v^\times\mathcal K'_{f,v}.$

If $K_{f,v}$ is maximal, namely if
$\mathcal K_{f,v}=\mathcal O_{D,v}^{\times}$, then
$K_{T,R,v}
        =
        R_v^\times/\mathcal O_v^\times.$
For arbitrary fixed finite level one has instead
\begin{equation}\label{e:kt}
    K_{T,R,v}
        \subseteq
        R_v^\times/\mathcal O_v^\times
        \end{equation}
and
\[
        \bigl[
        R_v^\times/\mathcal O_v^\times:K_{T,R,v}
        \bigr]
        \ll_{K_{f,v}}1.
\tag{*'}
\]
Here $\operatorname{vol}$ denotes volume with respect to a fixed Haar measure on the
local torus
$T_E(F_v)$
Both $K_{T,R,v}$ and $R_v^\times/\mathcal O_v^\times$ are compact open subgroups of
this torus.  Since
$K_{T,R,v}\subset R_v^\times/\mathcal O_v^\times$
with index bounded in terms of the fixed local level $K_{f,v}$, Haar
translation-invariance gives
$\operatorname{vol}(R_v^\times/\mathcal O_v^\times)
        =
        [R_v^\times/\mathcal O_v^\times:K_{T,R,v}]\,
        \operatorname{vol}(K_{T,R,v})$,
i.e.,
\[
        \operatorname{vol}(K_{T,R,v})
        \asymp_{K_{f,v}}
        \operatorname{vol}(R_v^\times/\mathcal O_v^\times).
\tag{*''}
\]

\begin{lemma}
\label{lem:local-fixed-level-comparison}
With the notation above, assume that $K_{f,v}$ is the image of an open compact
subgroup $\mathcal K_{f,v}\subset\mathcal O_{D,v}^{\times}$.

There is a constant $C_v=C(K_{f,v})$ such that, uniformly in $E_v$, in
the embedding $E_v\hookrightarrow D_v$, and in the order $R_v$,
$\bigl[
R_v^\times/\mathcal O_v^\times:K_{T,R,v}
\bigr]\le C_v$.
Equivalently,
\[
        \operatorname{vol}(K_{T,R,v})
        \asymp_{K_{f,v}}
        \operatorname{vol}(R_v^\times/\mathcal O_v^\times).
\]
\end{lemma}

\begin{proof}
Let
\[
        \mathcal O'_{D,v}
        =
        \widetilde g_{R,v}\mathcal O_{D,v}\widetilde g_{R,v}^{-1},
        \qquad
        R_v=E_v\cap\mathcal O'_{D,v}.
\]
The inverse image of $K_{T,R,v}$ in $E_v^\times$ is
$\widetilde K_{T,R,v}
        =
        E_v^\times\cap F_v^\times
        \widetilde g_{R,v}\mathcal K_{f,v}\widetilde g_{R,v}^{-1}$.
Since $\mathcal K_{f,v}\subset\mathcal O_{D,v}^{\times}$, we have
\[
        \widetilde K_{T,R,v}
        \subset
        E_v^\times\cap F_v^\times(\mathcal O'_{D,v})^\times
        =
        F_v^\times R_v^\times.
\]
After quotienting by $F_v^\times$, this gives
\[
        K_{T,R,v}\subset R_v^\times/\mathcal O_v^\times.
\]

Since $\mathcal K_{f,v}$ is open, there exists $m_v\ge1$, depending only on
$K_{f,v}$, such that
\[
        1+\mathfrak p_v^{m_v}\mathcal O_{D,v}
        \subset
        \mathcal K_{f,v}.
\]
Conjugating gives
$ 1+\mathfrak p_v^{m_v}\mathcal O'_{D,v}
        \subset
        \widetilde g_{R,v}\mathcal K_{f,v}\widetilde g_{R,v}^{-1}$.
Intersecting with $E_v^\times$, we obtain
$1+\mathfrak p_v^{m_v}R_v
        \subset
        \widetilde K_{T,R,v}$.
Thus the image of $1+\mathfrak p_v^{m_v}R_v$ in
$R_v^\times/\mathcal O_v^\times$ is contained in $K_{T,R,v}$.

Write
$ R_v=\mathcal O_v+\mathfrak p_v^{c_v}\mathcal O_{E_v}$.
If $c_v\ge1$, then
$R_v^\times =\mathcal O_v^\times\bigl(1+\mathfrak p_v^{c_v}\mathcal O_{E_v}\bigr)$,
and
\[
        1+\mathfrak p_v^{m_v}R_v
        \supset
        1+\mathfrak p_v^{m_v+c_v}\mathcal O_{E_v}.
\]
Therefore
\[
        \bigl[
        R_v^\times/\mathcal O_v^\times:K_{T,R,v}
        \bigr]
        \le
        \bigl[
        1+\mathfrak p_v^{c_v}\mathcal O_{E_v}:
        1+\mathfrak p_v^{c_v+m_v}\mathcal O_{E_v}
        \bigr]
        \ll_{F_v,m_v}1.
\]
The case $c_v=0$ is handled similarly, since $\mathcal O_{E_v}^{\times}/
        \mathcal O_v^\times(1+\mathfrak p_v^{m_v}\mathcal O_{E_v})$
has cardinality bounded in terms of $F_v$ and $m_v$.  
\end{proof}

The distinguished place $\infty$ is not included in the products below.  Its local
factor is fixed by the conditions
\[
        E_\infty\simeq F_{\infty,\star}
        \qquad
        \text{and}
        \qquad
        U_\infty=T_0(F_\infty),
\]
and is absorbed later in the global Waldspurger estimate.

\begin{lemma}
\label{lem:packet-size-lower-bound}
For every $\varepsilon>0$, one has
\[
        \#\calp_{E,R}
        \gg_{F,U,\varepsilon}
        D(E,R)^{1/2-\varepsilon}.
\]
\end{lemma}

\begin{proof}
Let
\[
        \calp_R^{\max}
        :=
        T_E(F)\backslash T_E(\mathbb A_{F,f})/
        \bigl(\widehat R^{\,\times}/\widehat A^{\,\times}\bigr).
\]
By \eqref{e:kt} we have $K_{T,R}\subset \widehat R^{\,\times}/\widehat A^{\,\times}$.
Hence there is a natural surjection
$\calp_{E,R}\twoheadrightarrow \calp_R^{\max}$.

Consider
\[
        \Pic(R):=
        E^\times\backslash\mathbb A_{E,f}^{\times}/\widehat R^{\,\times}
        \qquad
        \text{and}
        \qquad
        h(R):=\#\Pic(R).
\]
The natural map
$\Pic(R)\rightarrow \calp_R^{\max}$
has kernel contained in the image of
$\Pic(A)=F^\times\backslash\mathbb A_{F,f}^{\times}/\widehat A^{\,\times}$.
Since $\Pic(A)$ depends only on $F$,
\[
        \#\calp_R^{\max}\gg_F h(R).
\]
The function field class number formula for the order $R$ (see, for instance, \cite[Lemma~2.3]{guo-wei-class-number}) together with RH as in \cite[Thm.~3.3.1]{deligne-weilII}, gives
$h(R)\gg_{F,\varepsilon}D(E,R)^{1/2-\varepsilon}$.
Therefore
\[
        \#\calp_{E,R}
        \gg_{F,U,\varepsilon}
        D(E,R)^{1/2-\varepsilon}.
\]
\end{proof}

\begin{lemma}
\label{lem:elmv-local-building-integral-n2}
Let $k$ be a non-archimedean local field, let $\mathcal A/k$ be a quadratic
\'etale algebra, and let
$\iota\colon \mathcal A\hookrightarrow M_2(k)$
be a $k$-algebra embedding.  Let $g\in \PGL_2(k)$, and put
$\Lambda(g)
        :=
        \mathcal A\cap gM_2(\mathcal O_k)g^{-1}$.
Define
\[
        \Delta(g)
        :=
        \frac{\disc(\Lambda(g))}
             {\disc(\mathcal O_{\mathcal A})},
        \qquad
        b(g):=\operatorname{ord}_k\Delta(g).
\]
Let $\Xi_k$ be the Harish--Chandra spherical function of $\PGL_2(k)$ attached to
$\PGL_2(\mathcal O_k)$.  Then there exist constants
$A_0:=A_0(\pi_v,f_v,K_{f,v})\geq 0$ and $\eta_0:=\eta_0(\pi_v,f_v,K_{f,v})>0$
 such that
\[
        \int_{\mathcal A^\times/k^\times}
        \Xi_k(g^{-1}\iota(t)g)\,d^\times t
        \ll
        (b(g)+1)^{A_0}
        \vol(\mathcal O_{\mathcal A}^{\times}/\mathcal O_k^\times)
        \Delta(g)^{-\eta_0}.
\tag{ELMV}
\]
The implied constant is uniform in $k$, in the quadratic \'etale algebra
$\mathcal A/k$, in the embedding $\iota$, and in $g$. 
\end{lemma}

\begin{proof}
This amount just to the $n=2$ specialization of the local harmonic analysis in
\cite[\S7--\S9]{ELMV3}.

\end{proof}

\begin{remark}[Residue characteristic $2$]{\em
\label{rem:elmv-small-characteristic}
The discriminants in Lemma~\ref{lem:elmv-local-building-integral-n2} are trace-dual  discriminants \cite[\S9.4--\S9.5]{ELMV3}, hence defined for
every separable quadratic \'etale algebra $\mathcal A/k$, including wild ramification
in residue characteristic $2$. Residue characteristic enters only through  \cite[Lemma~9.12]{ELMV3},
\[
        \inf_{t\in\mathcal A^\times}
        \operatorname{dist}\bigl(N_0,t\,\iota^{-1}N_{\mathcal A}\bigr)
        \geq
        c_n\log\Delta(g)-O_n(1),
\]
where for $n=2$ in  characteristic $2$ the additive $O_n(1)$ cannot be
discarded. The leading coefficient $c_n>0$ is, however, unaffected; the lost term
depends only on $n$ and, passing through the Harish--Chandra bound
\cite[\S7.7,~(20)]{ELMV3} and the shell estimates \cite[Lemmas~9.13--9.14]{ELMV3},
rescales only the implied constant. Hence
Lemma~\ref{lem:elmv-local-building-integral-n2} holds in residue characteristic $2$
with the same exponent $\eta_0$.
}
\end{remark}

\begin{lemma}
\label{lem:split-local-matrix-coefficient-estimate}
Let $v\neq \infty$ be a finite place such that
\[
        D_v\simeq M_2(F_v).
\]
Let $f_v$ belong to the fixed finite-dimensional space of $K_{f,v}$-fixed
vectors occurring in $\pi_v$.  Then there exist constants
$A:=A(\pi_v,f_v,K_{f,v})\geq 0$ and $\eta:=\eta(\pi_v,f_v,K_{f,v})>0$ such that, uniformly in
$E_v/F_v$, $R_v$, the local embedding, and the packet character $\chi_v$,
one has
\[
        \alpha_v^\natural
        \bigl(
        f_{E,R,v},\overline f_{E,R,v};\chi_v
        \bigr)
        \ll_{f_v,K_{f,v}}
        (n_v(E,R)+1)^A
        \vol(\mathcal O_{E_v}^{\times}/\mathcal O_v^\times)
        q_v^{-\eta b_v(E,R)}.
\tag{S}
\]
For all but finitely many $v$, the constants are absolute and the normalized
unramified factor is equal to $1$ when $R_v=\mathcal O_{E_v}$.
\end{lemma}
\begin{proof}
Fix the isomorphism $D_v\simeq M_2(F_v)$ so that the fixed maximal order
$\mathcal O_{D,v}$ is identified with $M_2(\mathcal O_v)$. 
By normalization of the local period
\cite[\S1.2.3]{qiu}, $\alpha_v^\natural$ is bounded, up to local normalizing
$L$-factors bounded in this fixed unitary family, by
\[
        \int_{F_v^\times\backslash E_v^\times}
        \left|
        \left\langle
        \pi_v(g_{E,R,v}^{-1}tg_{E,R,v})f_v,f_v
        \right\rangle
        \right|\,d^\times t .
\tag{1}
\]
Since $D_v$ is split, $\pi_v$ is the local component of the $\GL_2$-representation
$\Pi$.  By \cite[Thm.~VI.10(i)]{lafforgue}, $\pi_v$ is tempered.  Hence the standard
$K$-finite matrix-coefficient bound gives
\[
        \left|
        \left\langle
        \pi_v(h)f_v,f_v
        \right\rangle
        \right|
        \ll_{f_v,K_{f,v}}
        \Xi_v(h),
\tag{2}
\]
where $\Xi_v$ is the Harish--Chandra spherical function of $\PGL_2(F_v)$.  Therefore
\begin{equation}
\label{e:8.3}
        \alpha_v^\natural
        \bigl(
        f_{E,R,v},\overline f_{E,R,v};\chi_v
        \bigr)
        \ll_{f_v,K_{f,v}}
        \int_{F_v^\times\backslash E_v^\times}
        \Xi_v(g_{E,R,v}^{-1}tg_{E,R,v})\,d^\times t .
\end{equation}

We now apply 
Lemma~\ref{lem:elmv-local-building-integral-n2}.  The associated local order is
$\Lambda_v
        =
        E_v\cap g_{E,R,v}M_2(\mathcal O_v)g_{E,R,v}^{-1}$.
Since $g_{E,R,v}=g_{R,v}$ at finite places, admissibility gives $\Lambda_v=R_v$.
Thus, with trace-dual discriminants,
\begin{equation}\label{e:8.4}
        \frac{\disc(\Lambda_v)}
             {\disc(\mathcal O_{E_v})}
        =
        q_v^{b_v(E,R)}.
\end{equation}
Combining \eqref{e:8.3}, Lemma~\ref{lem:elmv-local-building-integral-n2},
and \eqref{e:8.4}, we get
\[
        \alpha_v^\natural
        \bigl(
        f_{E,R,v},\overline f_{E,R,v};\chi_v
        \bigr)
        \ll_{f_v,K_{f,v}}
        (n_v(E,R)+1)^A
        \vol(\mathcal O_{E_v}^{\times}/\mathcal O_v^\times)
        q_v^{-\eta b_v(E,R)}.
\]
This proves $(S)$.

At unramified  places with $R_v=\mathcal O_{E_v}$, the normalization \cite[\S1.2.3]{qiu} gives
$\alpha_v^\natural
        \bigl(
        f_{E,R,v},\overline f_{E,R,v};\chi_v
        \bigr)=1$.
\end{proof}

\begin{lemma}
\label{lem:division-place-filtration-estimate}
Let $v\neq \infty$ be a place where $D_v$ is division. 
Let $E_v/F_v$ be a separable quadratic field embedded in $D$, and let
$ R_v=\mathcal O_v+\mathfrak f_{R,v}\mathcal O_{E_v}$
be the local order arising from an admissible fixed-level packet.  Write
$ \mathfrak f_{R,v}=\mathfrak p_v^{c_v}$.
Then $R_v=\mathcal O_{E_v}$.
In particular,
\[
        c_v=0,
        \qquad
        b_v(E,R)=0.
\tag{D1}
\]
Moreover, for every fixed $f_v\in\pi_v^{K_{f,v}}$,
\[
        \frac{
        \alpha_v^\natural
        \bigl(
        f_{E,R,v},\overline f_{E,R,v};\chi_v
        \bigr)
        }{
        \vol(K_{T,R,v})^2
        }
        \ll_{f_v,K_{f,v}}
        q_v^{a_v(E)/2}.
\tag{D2}
\]
The constants are uniform in $E_v/F_v$, in the embedding
$E_v\hookrightarrow D_v$, and in the packet character $\chi_v$.
\end{lemma}

\begin{proof}
Let $\mathcal O_{D_v}$ be the maximal order of  $D_v$.
Since $D_v/F_v$ is division, $\mathcal O_{D_v}$ is the unique maximal order.  Hence
\[
        g\mathcal O_{D_v}g^{-1}=\mathcal O_{D_v}
        \qquad(g\in D_v^\times).
\]
Since $D_v$ is division, its maximal order $\mathcal O_{D,v}$ is unique.  Hence
$g\mathcal O_{D,v}g^{-1}=\mathcal O_{D,v}$ for every $g\in D_v^\times$.  Therefore
the order cut out by the packet is
\[
        R_v
        =
        E_v\cap g\mathcal O_{D,v}g^{-1}
        =
        E_v\cap\mathcal O_{D,v}.
\]
This intersection is $\mathcal O_{E_v}$: an element of $E_v$ lying in
$\mathcal O_{D,v}$ is integral over $\mathcal O_v$, and conversely every element of
$\mathcal O_{E_v}$ is integral over $\mathcal O_v$ and hence lies in the unique
maximal order $\mathcal O_{D,v}$.  Thus $R_v=\mathcal O_{E_v}$. Therefore
$c_v=0$ and
$b_v(E,R)=0$
which proves $(D1)$.

For $(D2)$, by the definition of $\alpha_{\pi_v}^{\sharp}$ as in \cite[\S1.2.3]{qiu}, the factor $\alpha_v^\natural$ is the toric
matrix-coefficient integral multiplied by local normalizing $L$-factors.

Since $\chi_v$ is unitary, the local factor is bounded by the absolute toric
matrix-coefficient integral
\[
        \alpha_v^\natural
        \bigl(
        f_{E,R,v},\overline f_{E,R,v};\chi_v
        \bigr)
        \ll_{f_v,K_v}
        \int_{F_v^\times\backslash E_v^\times}
        \left|
        \left\langle
        \pi_v(t)f_{E,R,v},f_{E,R,v}
        \right\rangle
        \right|\,d^\times t .
\]
The representation $\pi_v$ is unitary, and $f_{E,R,v}$ is a translate of the fixed
vector $f_v$.  Hence
\[
        \left|
        \left\langle
        \pi_v(t)f_{E,R,v},f_{E,R,v}
        \right\rangle
        \right|
        \leq
        \|f_v\|^2 .
\]
Moreover $E_v/F_v$ is a field, so $F_v^\times\backslash E_v^\times$ is compact and its
valuation quotient has order at most $2$.  Therefore
$\vol(k^\times\backslash E_v^\times)
        \asymp_{F_v}
        \vol(\mathcal O_{E_v}^{\times}/\mathcal O_v^\times)$.

\begin{equation}\label{e:10.1}
        \alpha_v^\natural
        \bigl(
        f_{E,R,v},\overline f_{E,R,v};\chi_v
        \bigr)
        \ll_{f_v,K_v}
        \vol(\mathcal O_{E_v}^{\times}/\mathcal O_v^\times).
\end{equation}
By Lemma~\ref{lem:local-fixed-level-comparison} and $R_v=\mathcal O_{E_v}$,
\begin{equation}\label{e:10.2}
        \vol(K_{T,R,v})
        \asymp_{K_{f,v}}
        \vol(\mathcal O_{E_v}^{\times}/\mathcal O_v^\times).
\end{equation}

Finally,
\begin{equation}\label{e:10.3}
        \vol(\mathcal O_{E_v}^{\times}/\mathcal O_v^\times)
        \asymp_{F_v}
        q_v^{-a_v(E)/2}.
\end{equation}
Combining \eqref{e:10.1}, \eqref{e:10.2}, and \eqref{e:10.3} gives
\[
        \frac{
        \alpha_v^\natural
        \bigl(
        f_{E,R,v},\overline f_{E,R,v};\chi_v
        \bigr)
        }{
        \vol(K_{T,R,v})^2
        }
        \ll_{f_v,K_{f,v}}
        q_v^{a_v(E)/2}.
\]
This proves $(D2)$.
\end{proof}

Let
\[
        f=\otimes_v' f_v\in\pi^U
\]
be a fixed decomposable vector, where $\pi$ is a cuspidal automorphic
representation of $D^\times(\mathbb A_F)$ with trivial central character.  For
admissible CM data $(E,R)$, put
\[
        g_{E,R}:=z(1,g_{R,f}),
        \qquad
        f_{E,R}(g):=f(g\,g_{E,R}).
\]

\begin{proposition}
\label{prop:local-waldspurger-factor-estimate}
There exists $\delta=\delta(f,\pi,U)>0$ such that, for every admissible CM datum
$(E,R)$, every packet character $\chi$ attached to $(E,R)$, and every
$\varepsilon>0$,
\[
        \frac{
        \prod_{v\neq \infty}
        \alpha_v^\natural
        \bigl(f_{E,R,v},\overline f_{E,R,v};\chi_v\bigr)
        }{
        \operatorname{vol}(K_{T,R})^2\,\#\calp_{E,R}^{\,2}
        }
        \ll_{f,F,U,\varepsilon}
        D(E,R)^{-\delta+\varepsilon}.
\]
The implied constants are uniform as $\chi$ ranges over packet characters attached
to $(E,R)$.
\end{proposition}

\begin{proof}
For finite split places, Lemma~\ref{lem:split-local-matrix-coefficient-estimate}
gives, after increasing $A$ and decreasing $\eta>0$ once and for all,
\begin{equation}
\label{eq:split-alpha-local-bound}
        \alpha_v^\natural
        \bigl(f_{E,R,v},\overline f_{E,R,v};\chi_v\bigr)
        \ll_{f_v,K_{f,v}}
        (n_v(E,R)+1)^A
        \vol(\mathcal O_{E_v}^{\times}/\mathcal O_v^\times)
        q_v^{-\eta b_v(E,R)}.
\end{equation}
Here the constants are uniform in $v$: outside a fixed finite set the normalized
unramified factor is $1$ when $R_v=\mathcal O_{E_v}$.
By Lemma~\ref{lem:local-fixed-level-comparison} and the local order-volume formula,
\begin{equation}
\label{eq:toric-level-volume-comparison}
        \vol(K_{T,R,v})
        \asymp_{K_{f,v}}
        \vol(R_v^\times/\mathcal O_v^\times)
        \asymp_{K_{f,v}}
        \vol(\mathcal O_{E_v}^{\times}/\mathcal O_v^\times)
        q_v^{-b_v(E,R)/2}.
\end{equation}
Also
\begin{equation}
\label{eq:maximal-unit-volume-field-discriminant}
        \vol(\mathcal O_{E_v}^{\times}/\mathcal O_v^\times)
        \asymp_{F_v}
        q_v^{-a_v(E)/2}.
\end{equation}
Combining \eqref{eq:split-alpha-local-bound},
\eqref{eq:toric-level-volume-comparison}, and
\eqref{eq:maximal-unit-volume-field-discriminant}, we obtain
\begin{equation}
\label{eq:split-normalized-local-factor}
        \frac{
        \alpha_v^\natural
        \bigl(f_{E,R,v},\overline f_{E,R,v};\chi_v\bigr)
        }{
        \vol(K_{T,R,v})^2
        }
        \ll_{f_v,K_{f,v}}
        (n_v(E,R)+1)^A
        q_v^{a_v(E)/2}
        q_v^{(1-\eta)b_v(E,R)}.
\end{equation}

At finite division places, Lemma~\ref{lem:division-place-filtration-estimate} gives
$b_v(E,R)=0$ and
\begin{equation}
\label{eq:division-normalized-local-factor}
        \frac{
        \alpha_v^\natural
        \bigl(f_{E,R,v},\overline f_{E,R,v};\chi_v\bigr)
        }{
        \vol(K_{T,R,v})^2
        }
        \ll_{f_v,K_{f,v}}
        q_v^{a_v(E)/2}.
\end{equation}

Multiplying \eqref{eq:split-normalized-local-factor} over finite split places and
\eqref{eq:division-normalized-local-factor} over finite division places gives
\[
        \frac{
        \prod_{v\neq\infty}
        \alpha_v^\natural
        \bigl(f_{E,R,v},\overline f_{E,R,v};\chi_v\bigr)
        }{
        \vol(K_{T,R})^2
        }
        \ll_{f,U}
        \left(\prod_{v\neq\infty}q_v^{a_v(E)/2}\right)
        \left(\prod_{v\neq\infty}q_v^{(1-\eta)b_v(E,R)}\right)
        \prod_{v\neq\infty}(n_v(E,R)+1)^A.
\]
Recall
$\prod_{v\neq\infty}q_v^{a_v(E)}=D_E$ and $\prod_{v\neq\infty}q_v^{b_v(E,R)}=\Delta_R$. Since $n_v(E,R)$ is the coefficient of $v$ in the effective divisor
$\mathfrak D_{E,R}$, the bound for divisors on a fixed global function field
\cite[Lemma~5.8]{rosen-function fields}, gives, for every
$\varepsilon>0$,
\begin{equation}
\label{eq:divisor-factor-bound}
        \prod_{v\neq\infty}(n_v(E,R)+1)^A
        \ll_{F,A,\varepsilon}
        q^{\varepsilon\deg\mathfrak D_{E,R}}
        =
        D(E,R)^\varepsilon .
\end{equation}
we obtain
\begin{equation}
\label{eq:global-local-product-bound}
        \frac{
        \prod_{v\neq\infty}
        \alpha_v^\natural
        \bigl(f_{E,R,v},\overline f_{E,R,v};\chi_v\bigr)
        }{
        \vol(K_{T,R})^2
        }
        \ll_{f,U,\varepsilon}
        D_E^{1/2+\varepsilon}
        \Delta_R^{1-\eta+\varepsilon}.
\end{equation}

Combining \eqref{eq:global-local-product-bound} and the bound of Lemma~\ref{lem:packet-size-lower-bound}
(and renaming $\varepsilon$) gives
\begin{equation}
\label{eq:after-packet-normalization}
        \frac{
        \prod_{v\neq\infty}
        \alpha_v^\natural
        \bigl(f_{E,R,v},\overline f_{E,R,v};\chi_v\bigr)
        }{
        \vol(K_{T,R})^2\,\#\calp_{E,R}^{\,2}
        }
        \ll_{f,U,\varepsilon}
        D_E^{-1/2+\varepsilon}
        \Delta_R^{-\eta+\varepsilon}.
\end{equation}

Choose $0<\delta<\min\{1/2,\eta\}$.  Since
as $D_E^{-1/2}\Delta_R^{-\eta}
        \leq
        (D_E\Delta_R)^{-\delta}
        =
        D(E,R)^{-\delta}$
the stated estimate follows from \eqref{eq:after-packet-normalization}, after replacing $\varepsilon$ by a smaller positive number.
\end{proof}

\subsubsection{The function field Lindel\"of input}
\label{subsubsec:function field-lindelof}

The required central-value estimate follows from the Riemann hypothesis for
function field $L$-functions.

\begin{proposition}
\label{prop:function field-lindelof-input}
Let $\Pi$ be a fixed cuspidal automorphic representation of
$\GL_2(\mathbb A_F)$ with unitary central character.  Let $E/F$ vary through the quadratic extensions occurring in
the CM-packet construction, and let $\chi$ vary through the associated packet
characters.  Then, for every $\varepsilon>0$,
\[
        \left|
        L\!\left(\frac12,\Pi_E\otimes\chi\right)
        \right|
        \ll_{\Pi,F,\varepsilon}
        C(\Pi_E\otimes\chi)^\varepsilon .
\]
Moreover, for the CM data considered here, there exists a constant $A:=A(\Pi,F,U)>0$,  such that
\[
        C(\Pi_E\otimes\chi)
        \ll_{\Pi,F,U}
        D(E,R)^A .
\]
Consequently,
\[
        \left|
        L\!\left(\frac12,\Pi_E\otimes\chi\right)
        \right|
        \ll_{\Pi,F,U,\varepsilon}
        D(E,R)^\varepsilon .
\]
\end{proposition}

\begin{proof}
Let us recall the following standard notation. We write $\boxplus$ for the isobaric sum on general linear groups.  Thus, if
$\pi_i$ are automorphic representations of $\GL_{n_i}(\mathbb A_F)$, then
$\boxplus_i\pi_i$ denotes the isobaric automorphic representation of
$\GL_{\sum_i n_i}(\mathbb A_F)$ characterized, in particular, by the factorization
of standard and Rankin--Selberg $L$-functions:
\[
        L(s,\boxplus_i\pi_i)=\prod_i L(s,\pi_i),
        \qquad
        L(s,\rho\times\boxplus_i\pi_i)
        =
        \prod_i L(s,\rho\times\pi_i).
\]  For a unitary Hecke character
$\chi\colon E^\times\backslash\mathbb A_E^\times\rightarrow \mathbb C^\times$
let $\operatorname{AI}_{E/F}\chi$
denote its automorphic induction from $E$ to $F$.  Thus
$\operatorname{AI}_{E/F}\chi$ is the isobaric automorphic representation of
$\GL_2(\mathbb A_F)$ characterized by
\[
        \bigl(\operatorname{AI}_{E/F}\chi\bigr)_E
        \simeq
        \chi\boxplus\chi^\sigma ,
\]
where $\sigma$ is the nontrivial element of $\operatorname{Gal}(E/F)$.   It is cuspidal unless
$\chi$ is Galois-invariant, in which case it is an isobaric sum of two Hecke
characters of $F$.  This is the automorphic induction supplied by Lafforgue's
functoriality theorem \cite[Thm.~VII.1]{lafforgue}.

By compatibility of automorphic induction with quadratic base change and
Rankin--Selberg $L$-functions, again as in
\cite[Thm.~VII.1]{lafforgue}, one has
\begin{equation}
\label{eq:AI-rankin-identity}
        L(s,\Pi_E\otimes\chi)
        =
        L\bigl(s,\Pi\times\operatorname{AI}_{E/F}\chi\bigr).
\end{equation}
Write
\[
        \operatorname{AI}_{E/F}\chi
        =
        \boxplus_j \Sigma_j
\]
for its isobaric decomposition into unitary cuspidal automorphic representations
$\Sigma_j$ of some $\GL_{n_j}(\mathbb A_F)$, with $n_j\in\{1,2\}$ and
$\sum_j n_j=2$.  Then
\[
        L\bigl(s,\Pi\times\operatorname{AI}_{E/F}\chi\bigr)
        =
        \prod_j L(s,\Pi\times\Sigma_j).
\]
For each factor, Lafforgue's Riemann hypothesis for Rankin--Selberg $L$-functions
over global function fields  \cite[Thm.~VI.10(ii)]{lafforgue} applies.  Hence
the completed $L$-function in \eqref{eq:AI-rankin-identity} satisfies RH after
unitary normalization. Moreover, via Lafforgue's automorphic-to-Galois parametrization \cite[Thm.~VI.9(i)]{lafforgue}, the
Grothendieck--Ogg--Shafarevich formula
\cite[Exp.~X, \S7]{sga5} gives
\[
N\bigl(\Pi\times\operatorname{AI}_{E/F}\chi\bigr) \ll_{\Pi,F}
1+\log_q C\bigl(\Pi\times\operatorname{AI}_{E/F}\chi\bigr),
\]
where $N(\cdot)$ denotes the degree of the corresponding completed
$L$-polynomial.  The standard Littlewood bound under RH, in the form of
\cite[Thms.~5.17 and~5.19]{iwaniec-kowalski} (which extend almost verbatim to function fields), gives, for every
$\varepsilon>0$,
\[
        \left|
        L\!\left(\frac12,\Pi_E\otimes\chi\right)
        \right|
        \ll_{\Pi,F,\varepsilon}
        C\bigl(\Pi\times\operatorname{AI}_{E/F}\chi\bigr)^\varepsilon .
\]
We write this conductor as $C(\Pi_E\otimes\chi)$, using
\eqref{eq:AI-rankin-identity}.  This proves the first asserted estimate.

It remains to compare this conductor with the CM discriminant.  For a finite place
$v\neq\infty$, write
\[
        R_v=\mathcal O_v+\mathfrak p_v^{c_v}\mathcal O_{E_v}.
\]
Since $\chi_v$ is trivial on the inverse image of $K_{T,R,v}$ in $E_v^\times$,
Lemma~\ref{lem:local-fixed-level-comparison} gives integers
$m_v=m_v(K_{f,v})\geq 0$, equal to $0$ for all but finitely many $v$, such that
\[
        a_{E_v}(\chi_v)\ll_F c_v+m_v .
\]
Here $a_{E_v}(\chi_v)$ is the usual $E_v$-normalized conductor exponent; the
implicit constant only accounts for the ramification index
$e(E_v/F_v)\leq 2$.  In the split case
$E_v\simeq F_v\times F_v$, the same assertion is interpreted componentwise.
By the conductor-discriminant formula for local induction
\cite[Ch.~VI, \S2, Corollary to Proposition~4]{serre-local-fields}, applied to
$\operatorname{AI}_{E_v/F_v}\chi_v$, we obtain
\[
        a_v\bigl(\operatorname{AI}_{E/F}\chi\bigr)
        \ll_F
        a_v(E)+c_v+m_v .
\]
Summing over $v\neq\infty$, and using that the $m_v$'s are supported on the
fixed finite set of level places, gives
\[
        C\bigl(\operatorname{AI}_{E/F}\chi\bigr)
        \ll_{F,U}
        D_E^{A_1}\Delta_R^{A_1}
        =
        D(E,R)^{A_1}
\]
for some constant $A_1=A_1(F,U)>0$.

Finally, since $\Pi$ is fixed, the Rankin--Selberg conductor satisfies
\[
        C\bigl(\Pi\times\operatorname{AI}_{E/F}\chi\bigr)
        \ll_{\Pi,F}
        C\bigl(\operatorname{AI}_{E/F}\chi\bigr)^{A_2}
\]
for some $A_2=A_2(\Pi,F)>0$.  Therefore
$ C(\Pi_E\otimes\chi)=C\bigl(\Pi\times\operatorname{AI}_{E/F}\chi\bigr)
\ll_{\Pi,F,U} D(E,R)^A$
for some $A=A(\Pi,F,U)>0$.  Substituting this into the preceding
Littlewood bound, and replacing $\varepsilon$ by $\varepsilon/A$, we conclude.
\end{proof}

\begin{remark}
\label{rem:function field-vs-char-zero}{\em
This is an important point at which the function field argument differs from the classical  strategy.  In characteristic zero one needs a genuine subconvex
estimate for the Rankin--Selberg $L$-function appearing after Waldspurger.  Over
global function fields, the required Lindel\"of-strength estimate follows from the
Riemann hypothesis for the relevant $L$-functions.}
\end{remark}
\subsubsection{Cuspidal Weyl sum decay}
\label{subsubsec:cuspidal-weyl-sum-decay}

\begin{lemma}
\label{l:cuspidal-weyl-sum-decay}
Let $f\in\pi^U$, where $\pi$ is a cuspidal automorphic representation of
$D^\times(\mathbb A_F)$ with trivial central character.  Then, for every packet
character $\chi$,
\[
        W(f;E,R,\chi)\rightarrow 0
\qquad
\text{as}
\qquad
        \deg\mathfrak D_{E,R}\rightarrow\infty .
\]
\end{lemma}

\begin{proof}
It is enough to prove the estimate for decomposable vectors
$f=\otimes_v' f_v\in\pi^U$.
Indeed, for the fixed factorizable level $U=\prod_v U_v$, one has
$\pi^U=\bigotimes_v'\pi_v^{U_v}$.
Thus every $f\in\pi^U$ is a finite linear combination of decomposable $U$-fixed
vectors.  Since the Weyl sum is linear in $f$, the general case follows termwise.
We therefore assume that $f$ is decomposable.  Since the Weyl sum is
linear in $f$, the general case follows termwise.  We therefore assume that $f$
is decomposable.
By \eqref{eq:compare-weyl-period},
\[
        P_\chi^D(f_{E,R})
        =
        \frac{\vol(K_T)}{w_{E,R}}\,
        \#\calp_{E,R}\,
        W(f;E,R,\chi).
\]
Applying Waldspurger's formula \eqref{eq:waldspurger-qiu} to $f_{E,R}$ gives
\[
                |W(f;E,R,\chi)|^2
        =
        \frac{w_{E,R}^{\,2}}
        {\vol(K_T)^2\,\#\calp_{E,R}^{\,2}}
        \cdot
        \frac{
        L(2,1_F)\,
        L\!\left(\frac12,\Pi_E\otimes\chi\right)
        }{
        2\,L(1,\Pi,\operatorname{Ad})
        }
        \prod_v
        \alpha_v^\natural
        \bigl(f_{E,R,v},\overline f_{E,R,v};\chi_v\bigr).
\]
The factors $L(2,1_F)$ and $L(1,\Pi,\operatorname{Ad})$ are fixed and nonzero.
By Proposition~\ref{prop:function field-lindelof-input},
\[
        L\!\left(\frac12,\Pi_E\otimes\chi\right)
        \ll_{f,\varepsilon}
        D(E,R)^\varepsilon.
\]
At the distinguished place $\infty$, the extension
$E_\infty\simeq F_{\infty,\star}$
and the subgroup
$U_\infty=T_\star(F_\infty)$
are fixed. At the distinguished place $\infty$, we have
\[
        g_{E,R,\infty}=z_\infty,
        \qquad
        U_\infty=T_\star(F_\infty),
        \qquad
        z_\infty^{-1}T_E(F_\infty)z_\infty=T_\star(F_\infty).
\]
Hence $T_E(F_\infty)\cap g_{E,R,\infty}U_\infty g_{E,R,\infty}^{-1}
=T_E(F_\infty)$,
whose volume is fixed.  Moreover $\chi_\infty=1$, and after conjugating by
$z_\infty$ the local toric integral defining
$\alpha_\infty^\natural(f_{E,R,\infty},\overline f_{E,R,\infty};\chi_\infty)$
is an integral over the fixed compact torus $T_\star(F_\infty)$ against the fixed
$T_\star(F_\infty)$-invariant vector $f_\infty$.  Thus the infinite local factor and the
infinite toric volume are fixed nonzero constants, depending only on $f$ and $U$.
Therefore Proposition~\ref{prop:local-waldspurger-factor-estimate} gives
\[
        \frac{
        \prod_v
        \alpha_v^\natural
        \bigl(f_{E,R,v},\overline f_{E,R,v};\chi_v\bigr)
        }{
        \vol(K_T)^2\,\#\calp_{E,R}^{\,2}
        }
        \ll_{f,\varepsilon}
        D(E,R)^{-\delta+\varepsilon}.
\]
Hence
$|W(f;E,R,\chi)|^2
        \ll_{f,\varepsilon}
        D(E,R)^{-\delta+\varepsilon}$.
\end{proof}

\subsubsection{The one-dimensional spectrum}
\label{subsubsec:reduced-norm-component-characters}

The cuspidal estimate above does not treat the one-dimensional automorphic spectrum.
This remaining part is finite and is controlled by the reduced norm.

Let
\[
        \widetilde U\subset D^\times(\mathbb A_F)
\]
be the inverse image of
$\subset G(\mathbb A_F)=D^\times(\mathbb A_F)/\mathbb A_F^\times$.
Define
\[
        \mathfrak X_U
        :=
        \left\{
        \eta\colon F^\times\backslash\mathbb A_F^\times\to\{\pm1\}
        \;:\;
        \eta\bigl(\nrd(\widetilde U)\bigr)=1
        \right\}.
\]
Thus $\mathfrak X_U$ is the finite group of quadratic Hecke characters of $F$
which are trivial on the reduced norms of the level subgroup $U$.  

For $\eta\in\mathfrak X_U$, define
\[
        \xi_\eta\colon G(F)\backslash G(\mathbb A_F)\rightarrow \mathbb C^\times
\qquad
g\mapsto \eta\bigl(\nrd(\widetilde g)\bigr)
\]
where $\widetilde g\in D^\times(\mathbb A_F)$ is any lift of $g$.  This is
well-defined, since replacing $\widetilde g$ by $a\widetilde g$, with
$a\in\mathbb A_F^\times$, changes the reduced norm by $a^2$, and $\eta$ is
quadratic.

\begin{lemma}
\label{lem:one-dimensional-spectrum}
The map
$\eta\mapsto \xi_\eta$
identifies $\mathfrak X_U$ with the group of right $U$-invariant
one-dimensional automorphic representations occurring in $L^2(X_U)$.
\end{lemma}

\begin{proof}
The map $\eta\mapsto \xi_\eta$ is well defined: $\xi_\eta$ is right $U$-invariant
because $\eta$ is trivial on $\nrd(\widetilde U)$.

Conversely, let $\xi$ be a one-dimensional automorphic character of
$G(\mathbb A_F)$, trivial on $G(F)$ and right $U$-invariant.  Pulling $\xi$
back to $D^\times(\mathbb A_F)$, we obtain a character trivial on
$D^\times(F)$, on $\mathbb A_F^\times$, and on $\widetilde U$.

By the strong approximation for
$\operatorname{SL}_1(D)$, this character factors through the reduced norm
$\nrd\colon D^\times(\mathbb A_F)\to\mathbb A_F^\times$; see
\cite[Main Theorem~28.5.3]{voight-quaternion-algebras}.
Hence $\xi=\eta\circ\nrd$ for a Hecke character $\eta$ of $F$.
Since $\xi$ descends to $D^\times/F^\times$, it is trivial on scalar adeles,
so $\eta(a^2)=1$ for all $a\in\mathbb A_F^\times$.  Thus $\eta$ is quadratic.
Finally, right $U$-invariance gives $\eta(\nrd(\widetilde U))=1$, hence
$\eta\in\mathfrak X_U$.
\end{proof}

Now let $(E,R)$ be an admissible CM datum, and let
\[
        \eta_E\colon F^\times\backslash\mathbb A_F^\times\to\{\pm1\}
\]
be the quadratic character attached to $E/F$.  Recall that
$g_{E,R}:=z(1,g_{R,f})\in G(\mathbb A_F)$.
Let $\nrd(g_{E,R})$ denote the reduced norm of any lift of $g_{E,R}$ to
$D^\times(\mathbb A_F)$.

\begin{lemma}
\label{lem:component-character-averages}
For $\eta\in\mathfrak X_U$, one has
\[
        W(\xi_\eta;E,R)
        =
        \begin{cases}
        1, & \eta=1,\\[4pt]
        \eta_E\bigl(\nrd(g_{E,R})\bigr), & \eta=\eta_E,\\[4pt]
        0, & \eta\neq 1,\eta_E.
        \end{cases}
\]

\end{lemma}

\begin{proof}
For $x_{E,R}(t_f)=G(F)\,z(1,t_f)(1,g_{R,f})\,U$
we have $\xi_\eta(x_{E,R}(t_f))
        =
        \eta\bigl(\nrd(g_{E,R})\bigr)\,
        \eta\bigl(N_{E/F}(t_f)\bigr)$.
Therefore
\[
        W(\xi_\eta;E,R)
        =
        \eta\bigl(\nrd(g_{E,R})\bigr)
        \frac{1}{\#\calp_{E,R}}
        \sum_{[t_f]\in\calp_{E,R}}
        \eta\bigl(N_{E/F}(t_f)\bigr).
\]

The function $[t_f]\mapsto \eta\bigl(N_{E/F}(t_f)\bigr)$
is a character of the finite abelian group $\calp_{E,R} =
T_E(F)\backslash T_E(\mathbb A_{F,f})/K_{T,R}$.
Indeed, it is trivial on $T_E(F)$, because $\eta$ is a Hecke character of $F$;
it is trivial on $K_{T,R}$, because $K_{T,R}\subset g_{R,f}K_fg_{R,f}^{-1}$ and
$\eta$ is trivial on $\nrd(\widetilde K_f)$.

The average of a character of a finite abelian group is $0$, unless the character
is trivial, in which case the average is $1$.  Thus the displayed average is
nonzero precisely when $\eta\circ N_{E/F}$
is trivial on $T_E(\mathbb A_{F,f})$.  Since $\eta\in\mathfrak X_U$ and
$T_E(F_\infty)$ is conjugate to $T_\star(F_\infty)$, the same character is also
trivial on $T_E(F_\infty)$.  Hence the induced base-change character on
$E^\times\backslash\mathbb A_E^\times$
is trivial.

By global class field theory, the base-change character
$\eta\circ N_{E/F}$
is trivial if and only if the quadratic extension attached to $\eta$ becomes split
over $E$.  Since $E/F$ is quadratic, this happens exactly for
$\eta=1$ or $\eta=\eta_E$.
\end{proof}

\begin{corollary}
\label{cor:component-character-decay}
Let $(E_n,R_n)$ be a sequence of admissible CM data.  Then
\[
        W(\xi_\eta;E_n,R_n)\rightarrow 0
\]
for every nontrivial $\eta\in\mathfrak X_U$ if and only if
$\eta_{E_n}\notin \mathfrak X_U-\{1\}$
for all sufficiently large $n$.
\end{corollary}

\begin{proof}
By Lemma~\ref{lem:component-character-averages}, for fixed nontrivial
$\eta\in\mathfrak X_U$ one has
$W(\xi_\eta;E_n,R_n)=0 $ unless $\eta=\eta_{E_n}$.
If $\eta=\eta_{E_n}$, then
$W(\xi_\eta;E_n,R_n)
=\eta_{E_n}\bigl(\nrd(g_{E_n,R_n})\bigr)$,
which has absolute value $1$.  Hence
$W(\xi_\eta;E_n,R_n)\to0$ if and only if $\eta=\eta_{E_n}$ occurs only finitely
many times.  Since $\mathfrak X_U$ is finite, this holds for every nontrivial
$\eta\in\mathfrak X_U$ if and only if
$\eta_{E_n}\notin\mathfrak X_U-\{1\}$
for all sufficiently large $n$.
\end{proof}

\subsection{Equidistribution in the compact case}
\label{subsec:equidistribution-compact}

We now prove the equidistribution theorem in the compact Drinfeld--Stuhler case.
Thus $D$ is a quaternion division algebra over $F$, split at $\infty$, and
\[
        G=D^\times/F^\times .
\]
Then $G$ is anisotropic over $F$, so
\[
        X_U=G(F)\backslash G(\mathbb A_F)/U,
        \qquad
        U=T_\star(F_\infty)K_f,
\]
is compact.  In particular, the automorphic spectrum is discrete; there is no
Eisenstein contribution.

We keep the notation of
\S\ref{subsubsec:canonical-measure-CM-packets} and
\S\ref{subsec:common-analytic-input}.  In particular, $\mu_U$ is the normalized
quotient Haar measure on $X_U$, and $\mu_{E,R}$ is the normalized CM-packet
measure attached to $(E,R)$.

\begin{theorem}
\label{thm:compact-adic-duke}
Assume that $D$ is division.  
Let $(E_n,R_n)$ be a sequence of admissible CM data
satisfying
\[
        (E_n)_\infty\simeq F_{\infty,\star}
        \qquad
        \text{and}
        \qquad
        \deg \mathfrak D_{E_n,R_n}\rightarrow \infty .
\]
Assume moreover that
$\eta_{E_n}\notin \mathfrak X_U-\{1\}$
for all sufficiently large $n$,
where $\mathfrak X_U=\mathfrak X_{U_\star}$ is the finite reduced-norm component-character group defined
in \S\ref{subsubsec:reduced-norm-component-characters}.

Let $\mu_{E_n,R_n,\star}$
be the corresponding normalized CM-packet measures on $X_U=X_{U_\star}$.  Then
$(\mu_{E_n,R_n,\star})$
weakly-$*$ converges to $\mu_U$  on $X_U=X_{U_\star}$ for $n\to \infty $.  

\end{theorem}

\begin{remark}{\em
If $K_f$ is the image of the finite level subgroup
$K_I\subset D^\times(\mathbb A_{F,f})$ appearing in the analytic uniformization,
then the projectivized map
\[
        X_U=X_{U_\star}
        \rightarrow
        M_{\mathcal D,I}^{\operatorname{an}}(\mathbb C_\infty)/\Delta_I
\]
pushes the above equidistribution statement forward to the corresponding
projectivized analytic CM packets.  For an equidistribution statement on the full
level-$I$ analytic curve
$M_{\mathcal D,I}^{\operatorname{an}}(\mathbb C_\infty)$,
one works instead with the lifted quotient $\widetilde X_{I,\star}$.}
\end{remark}

\begin{remark}
\label{rem:finite-component-condition}{\em
The finite-component condition
$\eta_{E_n}\notin \mathfrak X_U-\{1\}$ for $n\gg0$
is automatic if $\deg \mathfrak D_{E_n/F}\rightarrow\infty$
because $\mathfrak X_U=\mathfrak X_{U_\star}$ is finite and depends only on the fixed level $U=U_\star$.
It is not automatic if the field $E_n$ is fixed and only the conductor of the order
$R_n$ tends to infinity.  In that case full equidistribution on $X_U$ can fail
when $\eta_{E_n}\in\mathfrak X_U$.
}
\end{remark}

Let us gather the well-known Weyl criterion
for our compact setting.

\begin{lemma}
\label{lem:weyl-criterion-compact}
We have that
$\mu_{E_n,R_n}$ weakly-$*$ converges to $\mu_U$ if and only if
$W(\varphi;E_n,R_n)\rightarrow 0$
for every locally constant complex-valued function $\varphi$ on $X_U$ satisfying
$\int_{X_U}\varphi\,d\mu_U=0$.
\end{lemma}

\begin{proof}
Since $X_U$ is compact, weak convergence of probability measures is equivalent to
convergence against continuous functions.  Locally constant functions are dense in
$C(X_U)$, because $X_U$ is a compact totally disconnected adelic quotient.  Since
all measures involved have total mass $1$, it is enough to test functions with zero
$\mu_U$-mean.
\end{proof}

\subsubsection{Spectral reduction in the compact case}
\label{subsubsec:spectral-reduction-compact}

Because $\PGL_1(D)$ is anisotropic over $F$, the quotient
\[
        [G]:=G(F)\backslash G(\mathbb A_F)
\]
is compact.  Hence the right regular representation on $L^2([G])$ has purely
discrete spectrum:
\[
        L^2([G])
        =
        \widehat{\bigoplus}_{\pi\in\mathcal A(G)}
        m(\pi)\,\pi ,
\]
where $\mathcal A(G)$ denotes the irreducible automorphic representations occurring
in $L^2([G])$, and $m(\pi)<\infty$ is the corresponding multiplicity.  Since $U$ is
compact, averaging over $U$ projects onto right $U$-invariants, and therefore
\[
        L^2(X_U)
        =
        L^2([G])^U
        =
        \widehat{\bigoplus}_{\pi\in\mathcal A(G)}
        m(\pi)\,\pi^U .
\]
The trivial representation gives the constants.  The remaining one-dimensional
representations are exactly the component characters described in
Lemma~\ref{lem:one-dimensional-spectrum}.  Recall that cuspidality is defined by vanishing of constant terms along proper
$F$-parabolic subgroups \cite[\S4.6]{borel-jacquet}.  Since
$\PGL_1(D)$ is anisotropic, it has no proper $F$-parabolic subgroups
\cite[Exp.~XXVI, Cor.~6.12]{sga3-xxvi}.  Thus the cuspidality condition is vacuous
on the compact quotient.  We therefore separate the one-dimensional automorphic
characters and apply the cuspidal estimates only to the remaining non-character
discrete spectrum.

\begin{lemma}
\label{l:spectral-reduction-compact}
Assume that
\begin{enumerate}
\item for every nontrivial $\eta\in\mathfrak X_U$,
$W(\xi_\eta;E_n,R_n)\rightarrow 0$;

\item for every non-one-dimensional cuspidal automorphic representation
$\pi\subset L^2([G])$ and every $f\in\pi^U$,
$W(f;E_n,R_n)\rightarrow 0$.
\end{enumerate}
Then
$\mu_{E_n,R_n}$ weakly-$*$ converges to $\mu_U$.
\end{lemma}

\begin{proof}
By Lemma~\ref{lem:weyl-criterion-compact}, it is enough to prove
$W(\varphi;E_n,R_n)\rightarrow 0$
for every locally constant $\varphi$ on $X_U$ with zero $\mu_U$-mean.

View $\varphi$ as a right $U$-invariant function on $[G]$.  Since $\varphi$ is
locally constant, it is right invariant under some compact open subgroup
$J\subset G(\mathbb A_F)$.  Since $G$ is anisotropic, the double quotient
$G(F)\backslash G(\mathbb A_F)/J$
is finite.  Hence $L^2([G])^J$ is finite dimensional, and $\varphi$ has a finite
spectral expansion, which splits as
\[
        \varphi
        =
        \varphi_{\mathrm{triv}}
        +
        \varphi_{\mathrm{comp}}
        +
        \varphi_{\mathrm{cusp}}.
\]
The zero-mean condition gives $\varphi_{\mathrm{triv}}=0$.  By
Lemma~\ref{lem:one-dimensional-spectrum}, $\varphi_{\mathrm{comp}}$ is a finite
linear combination of the nontrivial characters $\xi_\eta$ with
$\eta\in\mathfrak X_U-\{1\}$.  The remaining term
$\varphi_{\mathrm{cusp}}$ is a finite linear combination of vectors in
non-one-dimensional cuspidal representations.  Assumptions $(1)$ and $(2)$ therefore
give $W(\varphi;E_n,R_n)\rightarrow 0$.
By Lemma~\ref{lem:weyl-criterion-compact} we conclude.
\end{proof}

\begin{proof}[Proof of Theorem~\ref{thm:compact-adic-duke}]
The finite component condition
$\eta_{E_n}\notin\mathfrak X_U-\{1\}$ for $n\gg0$
implies, by Corollary~\ref{cor:component-character-decay}, that
\[
        W(\xi_\eta;E_n,R_n)\rightarrow 0
\]
for every nontrivial component character $\xi_\eta$.

The non-one-dimensional cuspidal contribution is handled by
Lemma~\ref{l:cuspidal-weyl-sum-decay}, applied with the trivial packet character
$\chi=1$.  Hence the two hypotheses of
Lemma~\ref{l:spectral-reduction-compact} are satisfied.  Therefore
we obtain the desired equidistribution.
\end{proof}

\subsection{Equidistribution in the non-compact case}
\label{subsec:equidistribution-noncompact}

We now assume that
\[
        D\simeq M_2(F),
        \qquad
        G=\PGL_2 .
\]
Then
\[
        X_U=\PGL_2(F)\backslash\PGL_2(\mathbb A_F)/U,
        \qquad
        U=T_\star(F_\infty)K_f,
\]
has finite volume but is not compact.  Thus the automorphic spectrum contains,
besides constants and cuspidal representations, the continuous spectrum generated by
Eisenstein series.

The Weyl sums, CM packets, packet measures, component characters, and cuspidal
Weyl-sum estimates are those of
\S\ref{subsec:common-analytic-input}.  The only new issue in the split case is the
Eisenstein contribution, which is the analytic expression of non-escape of mass.

\subsubsection{Statement of the theorem}
\label{subsubsec:split-duke-statement}

\begin{theorem}
\label{thm:split-adic-duke}
Assume that $D\simeq M_2(F)$, so that $G=\PGL_2$.  Let
$(E_n,R_n)$ be a sequence of admissible CM data  satisfying
\[
        (E_n)_\infty\simeq F_{\infty,\star}
        \qquad
        \text{and}
        \qquad
        \deg\mathfrak D_{E_n,R_n}\rightarrow\infty .
\]
Assume moreover that
$\eta_{E_n}\notin \mathfrak X_U-\{1\}$
for all sufficiently large $n$,
where $\mathfrak X_U=\mathfrak X_{U_\star}$ is the finite reduced-norm component-character group defined
in \S\ref{subsubsec:reduced-norm-component-characters}.

Let
$\mu_{E_n,R_n,\star}$
be the corresponding normalized CM-packet measures on
\[
        X_U=\PGL_2(F)\backslash\PGL_2(\mathbb A_F)/U.
\]
Then
$\mu_{E_n,R_n,\star}$ weakly-$*$ converges to $\mu_U$ on $X_U$ as $n\to\infty$. 
\end{theorem}

Let us gather also the non-compact version of the Weyl criterion.
Set
\[
        \mathcal T_U:=C_c^\infty(X_U)+\mathbb C\cdot 1 .
\]

\begin{lemma}
\label{lem:weyl-criterion-noncompact}
  Suppose that
\[
        W(\varphi;E_n,R_n)\rightarrow 0
\]
for every $\varphi\in\mathcal T_U$
satisfying $\int_{X_U}\varphi\,d\mu_U=0$.
Then
$\mu_{E_n,R_n}$ weakly-$*$ converges to $\mu_U$ on $X_U$.
\end{lemma}

\begin{proof}
Let $h\in C_c^\infty(X_U)$ and put
\[
        \varphi:=h-\int_{X_U}h\,d\mu_U .
\]
Then $\varphi\in\mathcal T_U$ has $\mu_U$-mean zero, so by hypothesis
\[
        \int_{X_U}h\,d\mu_n-\int_{X_U}h\,d\mu_U
        =
        W(\varphi;E_n,R_n)
        \rightarrow 0.
\]
Thus $\mu_{E_n,R_n}$ weakly-$*$ converges  to $\mu_U$, first against $C_c^\infty(X_U)$ and hence against $C_c(X_U)$ by density.  
\end{proof}

\subsubsection{Spectral reduction in the split case}
\label{subsubsec:spectral-reduction-noncompact}

By the function field spectral decomposition for $\PGL_2$
\cite[VI.2.7]{moeglin-waldspurger}, the $L^2$-spectrum of
\[
        [G]:=G(F)\backslash G(\mathbb A_F)
\]
decomposes as
\[
        L^2([G])
        =
        \mathbb C\cdot 1
        \oplus
        L^2_{\mathrm{comp}}
        \oplus
        L^2_{\mathrm{cusp}}
        \oplus
        L^2_{\mathrm{Eis}} .
\]
Here $L^2_{\mathrm{comp}}$ is generated by the nontrivial one-dimensional
automorphic characters, $L^2_{\mathrm{cusp}}$ is the cuspidal spectrum, and
$L^2_{\mathrm{Eis}}$ is the continuous Eisenstein spectrum.  After passing to
right $U$-invariants, the component-character part is finite-dimensional and is
spanned by the characters $\xi_\eta$ with
$\eta\in\mathfrak X_U-\{1\}$.
For a test function $h\in C_c^\infty(X_U)$, write
\[
        h
        =
        h_{\mathrm{triv}}
        +
        h_{\mathrm{comp}}
        +
        h_{\mathrm{cusp}}
        +
        h_{\mathrm{Eis}}
\]
for its spectral decomposition.

\begin{proposition}
\label{prop:spectral-reduction-noncompact}
Assume that

\begin{enumerate}
\item for every nontrivial component character $\xi_\eta$, with
$\eta\in\mathfrak X_U-\{1\}$,
$W(\xi_\eta;E_n,R_n)\to 0$;

\item for every $h\in C_c^\infty(X_U)$,
$W(h_{\mathrm{cusp}};E_n,R_n)\to 0$;

\item for every $h\in C_c^\infty(X_U)$,
$W(h_{\mathrm{Eis}};E_n,R_n)\to 0$.
\end{enumerate}

Then
$\mu_{E_n,R_n}$ weakly -$*$ converges to $\mu_U$.
\end{proposition}

\begin{proof}
By Lemma~\ref{lem:weyl-criterion-noncompact}, it is enough to test
$\varphi\in\mathcal T_U$ with zero $\mu_U$-mean.  Write
\[
        \varphi=h+c,
        \qquad
        h\in C_c^\infty(X_U).
\]
The constant $c$ contributes only to the trivial spectral component.  Hence the
nontrivial spectral part of $\varphi$ is
\[
        h_{\mathrm{comp}}+h_{\mathrm{cusp}}+h_{\mathrm{Eis}}.
\]
The zero-mean condition kills the trivial component.  The component part
$h_{\mathrm{comp}}$ is a finite linear combination of the nontrivial characters
$\xi_\eta$, so it is handled by $(1)$.  The cuspidal and Eisenstein parts are
handled by $(2)$ and $(3)$.  Therefore
\[
        W(\varphi;E_n,R_n)\to0.
\]
Lemma~\ref{lem:weyl-criterion-noncompact} gives the claim.
\end{proof}

\subsubsection{The Eisenstein contribution}
\label{subsubsec:eisenstein-contribution}

We now treat the Eisenstein contribution by direct unfolding.  The toric period of an Eisenstein
series unfolds to an abelian Hecke zeta integral.  The global $D_E$-aspect is
controlled by the function field Riemann hypothesis for the resulting abelian
$L$-function.  The local order aspect is controlled by the non-maximal order
estimate, and ramified local uniformity is supplied by the second estimate of
\cite[Proposition~9.9]{ELMV3}, proved using Tate's local functional equation and
the maximum modulus principle.

Let
\[
        P_{\mathrm{Eis}}\colon L^2(X_U)\rightarrow L^2_{\mathrm{Eis}}(X_U)
\]
be the orthogonal projection onto the Eisenstein spectrum.

In the following lemma, we write
$\omega=\otimes_v\omega_v$
for the unitary Hecke character of $F^\times\backslash\mathbb A_F^\times$
parametrizing the Eisenstein series.  We use the standard global and local Hecke
$L$-function notation:
\[
        L_F(s,\omega^2)=\prod_v L_{F_v}(s,\omega_v^2),
\qquad
\text{and}
\qquad
        L_E(s,\omega\circ N_{E/F})
        =
        \prod_v
        L_{E_v}(s,\omega_v\circ N_{E_v/F_v}).
\]
Here $N_{E/F}$ and $N_{E_v/F_v}$ are the global and local norm maps.  The embedding
$\iota\colon E\hookrightarrow M_2(F)$ is the fixed embedding used to realize $T_E$ inside $\PGL_2$, and $dt_v$ denotes the local quotient Haar measure on
$F_v^\times\backslash E_v^\times$ compatible with the product measure in the toric period.

\begin{lemma}
\label{lem:eis-unfolding}
Let $\mathcal E(g,\Phi,\omega)$ be a normalized Eisenstein series attached to a
$U$-invariant flat section $\Phi$ in a unitary principal series of
$\PGL_2(\mathbb A_F)$.  Put
\[
        \Phi_{v,E,R}(g):=\Phi_v(gg_{E,R,v})
\]
and
\[
        \mathcal P_{E,R}(\Phi,\omega)
        :=
        \int_{T_E(F)\backslash T_E(\mathbb A_F)}
        \mathcal E(tg_{E,R},\Phi,\omega)\,dt .
\]
Then
\begin{equation}
\label{eq:eis-unfolded}
        \mathcal P_{E,R}(\Phi,\omega)
        =
        \frac{
        L_E\!\left(\frac12,\omega\circ N_{E/F}\right)
        }{
        L_F(1,\omega^2)
        }
        \prod_v
        \mathcal Z_v^\natural(\Phi_{v,E,R},\omega_v),
\end{equation}
where
\[
        \mathcal Z_v^\natural(\Phi_{v,E,R},\omega_v)
        :=
        \frac{
        L_{F_v}(1,\omega_v^2)
        }{
        L_{E_v}\!\left(\frac12,\omega_v\circ N_{E_v/F_v}\right)
        }
        \int_{F_v^\times\backslash E_v^\times}
        \Phi_{v,E,R}(\iota(t_v))\,dt_v .
\]
For standard unramified local data, including $R_v=\mathcal O_{E_v}$ and spherical
$\Phi_v$, one has $\mathcal Z_v^\natural(\Phi_{v,E,R},\omega_v)=1$.
\end{lemma}

\begin{proof}
This is the adelic Hecke unfolding for a nonsplit torus over a global function
field; see \cite[Thm.~4.3]{lorscheid-toroidal}.  In the
region of absolute convergence, unfolding gives a Tate integral over the projective
idele quotient
$E^\times\mathbb A_F^\times\backslash\mathbb A_E^\times$.
Factoring this Tate integral locally and using the normalization of the Eisenstein
series gives \eqref{eq:eis-unfolded}.  The identity on the unitary axis follows by
meromorphic continuation.
\end{proof}

The Eisenstein estimate uses two distinct inputs.  First, the 
function field GRH gives the required Lindel\"of-strength bound for
the abelian Hecke $L$-value appearing after unfolding.  Second, at the local
places where the quadratic algebra, the embedding, or the inducing character is
ramified, we need the second estimate of \cite[Proposition~9.9]{ELMV3}.  We
record this second estimate separately, deriving it from Tate's local functional equation (see \cite{tate-thesis})
 and the maximum modulus principle (as suggested in \cite[p.39]{ELMV3}).

\begin{lemma}
\label{lem:tate-local-functional-equation}
Let $k$ be a non-archimedean local field, let $A/k$ be a quadratic \'{e}tale
algebra, and let $e_A=e\circ\operatorname{Tr}_{A/k}$.  For
$\nu\colon A^\times\to\mathbb C^\times$ unitary and $\Phi\in\mathcal S(A)$, put
\[
        Z(\Phi,\nu,s)
        :=
        \int_{A^\times}
        \Phi(x)\nu(x)|x|_A^s\,d^\times x .
\]
Let $\widehat\Phi$ be the Fourier transform with respect to $e_A$ and the
self-dual measure on $A$.  Then
\[
        \epsilon(A,\nu,s,e_A)
        \frac{Z(\Phi,\nu,s)}
             {L(A,\nu,s)}
        =
        \frac{
        Z(\widehat\Phi,\nu^{-1},1-s)
        }{
        L(A,\nu^{-1},1-s)
        },
\]
and both sides are holomorphic in $s$.  Moreover, on $\Re(s)=0$,
\begin{equation}\label{e:tate-epsilon}
        |\epsilon(A,\nu,s,e_A)|^{-1}
        \ll_{e,k}
        \bigl(\disc(\nu)\disc(A/k)\bigr)^{-1/2}.
\end{equation}
\end{lemma}

\begin{proof}
This is Tate's local functional equation for local zeta integrals, applied to the
quadratic \'etale $k$-algebra $A$.  We use it in the form quoted by 
\cite[Lemma~9.18]{ELMV3}, which refers to Tate \cite{tate-thesis}. 
\end{proof}

\begin{lemma}
\label{lem:tate-interpolation-estimate}
Let $k,A,e,\nu$ be as in Lemma~\ref{lem:tate-local-functional-equation}.  Let 
$\iota\colon A\rightarrow k^2$
be a $k$-linear trivialization, and consider the Bruhat--Schwartz function $\Psi\in\mathcal S(k^2)$. Set
\[
        \Phi(x):=\Psi(x^\iota).
\]
Let $N_A$ be the canonical norm on $A$, let $N_0$ be the standard norm on $k^2$,
and choose $h\in\GL_2(k)$ such that\footnote{With the convention of \cite{ELMV3} for the action of $\GL_2(k)$ on norms.}
$hN_0=\iota^{-1}N_A$. Then
\[
        |\det h|_k^{-1/2}
        \left|
        \int_{A^\times}
        \Psi(x^\iota)\nu(x)|x|_A^{1/2}\,d^\times x
        \right|
        \ll_{\Psi,e,k}
        \left(
        \frac{\vol(\mathcal O_A)}
             {\vol(\mathcal O_k)^2}
        \right)^{1/2}
        \bigl(\disc(\nu)\disc(A/k)\bigr)^{-1/4}.
\]
The implied constant is uniform for $\Psi$ in any fixed finite-dimensional space of
Bruhat--Schwartz functions.
\end{lemma}

\begin{proof}
We use the Tate local functional equation in the form recorded in
Lemma~\ref{lem:tate-local-functional-equation}, which is the local input isolated in
\cite[Lemma~9.18]{ELMV3}.  The interpolation below is the non-archimedean
maximum-modulus argument indicated after \cite[Proposition~9.9]{ELMV3}. Put
\[
        Z(s):=Z(\Phi,\nu,s).
\]
On $\Re(s)=1$, the trivial $L^1$-bound gives
\[
        |Z(s)|
        \ll_{\Psi,k}
        \frac{\iota_*\vol_A}{\vol_{k^2}}.
\]
On $\Re(s)=0$, Lemma~\ref{lem:tate-local-functional-equation} and the conductor formula for the
epsilon factor \eqref{e:tate-epsilon} give
\[
        \left|
        \frac{Z(s)}
             {L(A,\nu,s)/L(A,\nu^{-1},1-s)}
        \right|
        \ll_{\Psi,e,k}
        \bigl(\disc(\nu)\disc(A/k)\bigr)^{-1/2}.
\]
Equivalently, for
\[
        G(z):=L(A,\nu,s)^{-1}Z(s)
        \qquad
        \text{where}
        \quad
        z=q_k^{-s},
\]
we have the above two boundary bounds on the annulus
$q_k^{-1}\le |z|\le 1$.
Since $k$ is non-archimedean, the local zeta integral and the local $L$-factors are
rational functions of $z=q_k^{-s}$, and the normalized integral
$L(A,\nu,s)^{-1}Z(\Phi,\nu,s)$ is a Laurent polynomial in $z$
\cite[\S2]{tate-thesis}; in particular
$G(z)$ is holomorphic on the annulus $q_k^{-1}\le|z|\le1$. Hadamard's
three-circles theorem at $|z|=q_k^{-1/2}$ gives
\[
        |Z(1/2)|
        \ll_{\Psi,e,k}
        \left(
        \frac{\iota_*\vol_A}{\vol_{k^2}}
        \right)^{1/2}
        \bigl(\disc(\nu)\disc(A/k)\bigr)^{-1/4},
\]
using that $L(A,\nu,1/2)$ is uniformly bounded for unitary $\nu$ over the fixed
local field $k$.  Finally,
\[
        \frac{\iota_*\vol_A}{\vol_{k^2}}
        =
        \frac{\vol(\mathcal O_A)}
             {\vol(\mathcal O_k)^2}
        |\det h|_k.
\]
Multiplying by $|\det h|_k^{-1/2}$ gives the claim.
\end{proof}

We call a pair $(\Phi,\omega)$ a {\em level-$U$ Eisenstein datum} if $\omega$ is unitary and
$\Phi=\otimes_v'\Phi_v$ is a decomposable $U$-fixed section in the corresponding
principal series.  

\begin{proposition}
\label{prop:eis-local-subconvex}
There exists $\eta>0$, depending only on the fixed level $U$, such that for every
level-$U$ Eisenstein datum $(\Phi,\omega)$, every admissible CM datum $(E,R)$, and
every $\varepsilon>0$,
\[
        \frac{
        \prod_{v\neq\infty}
        \left|
        \mathcal Z_v^\natural(\Phi_{v,E,R},\omega_v)
        \right|
        }{
        \vol(K_{T,R})
        }
\ll_{\Phi,U,\varepsilon}D(E,R)^\varepsilon\Delta_R^{1/2-\eta}.
\]
\end{proposition}

\begin{proof}
For each finite place $v$, we set
\[
        \Lambda_v
        :=
        E_v\cap g_{E,R,v}M_2(\mathcal O_v)g_{E,R,v}^{-1},
        \qquad
        \Delta_v
        :=
        \frac{\disc(\Lambda_v)}
             {\disc(\mathcal O_{E_v})},
        \qquad
        r_v:=\operatorname{ord}_v\Delta_v .
\]
Consider
\[
        \nu_v:=\omega_v\circ N_{E_v/F_v}.
\]
Using the Schwartz-function realization of Eisenstein series in
\cite[\S10.1]{ELMV3} and the torus-unfolding formula
\cite[Lemma~10.4]{ELMV3}, the local factor
$\mathcal Z_v^\natural(\Phi_{v,E,R},\omega_v)$ is, after the normalization
already included in its definition, the $n=2$ local integral of
\cite[Prop.~9.9]{ELMV3}, with
\[
        k=F_v,\qquad
        A=E_v,\qquad
        \psi=\omega_v\circ N_{E_v/F_v}.
\]

Under this identification, the order occurring in {\em loc.~cit} is $\Lambda_v$,
and the relative discriminant appearing there is comparable to $\Delta_v$,
with constants depending only on $n=2$.

Outside a fixed finite set of places determined by $(\Phi,U)$, the normalized
local quotient is $1$ when $n_v(E,R)=0$.  At every other varying place one has
$n_v(E,R)\geq1$; hence, after increasing $A_0$, the uniform local implied
constants may be absorbed into the factor $(n_v(E,R)+1)^{A_0}$.
The first estimate of \cite[Prop.~9.9]{ELMV3}, together with
Lemma~\ref{lem:local-fixed-level-comparison} and the local order-volume formula,
therefore gives, after enlarging $A_0$ and decreasing $\eta>0$ once and for all,
\[
        \frac{
        \left|
        \mathcal Z_v^\natural(\Phi_{v,E,R},\omega_v)
        \right|
        }{
        \vol(K_{T,R,v})
        }
        \ll_{\Phi_v,K_{f,v}}(n_v(E,R)+1)^{A_0}\Delta_v^{1/2-\eta}.
\tag{1}
\]
Lemma~\ref{lem:tate-interpolation-estimate} gives the complementary bound
\[
        \left|
        \mathcal Z_v^\natural(\Phi_{v,E,R},\omega_v)
        \right|
        \ll_{\Phi_v,K_{f,v}}
        \left(
        \frac{\vol(\mathcal O_{E_v})}
             {\vol(\mathcal O_v)^2}
        \right)^{1/2}
        \bigl(
        \disc(\nu_v)\disc(E_v/F_v)
        \bigr)^{-1/4},
\tag{2}
\]
with the convention $\disc(E_v/F_v)=1$ in the split case.  With the self-dual
measure normalizations of \cite[\S8]{ELMV3}, relation
\cite[(32)]{ELMV3} gives
\[
        \left(
        \frac{\vol(\mathcal O_{E_v})}
             {\vol(\mathcal O_v)^2}
        \right)^{1/2}
        \disc(E_v/F_v)^{-1/4}
        \asymp_F
        \vol(\mathcal O_{E_v}^{\times}/\mathcal O_v^\times).
\]
Consequently, again by the local order-volume comparison,
\[
        \frac{
        \left|
        \mathcal Z_v^\natural(\Phi_{v,E,R},\omega_v)
        \right|
        }{
        \vol(K_{T,R,v})
        }
        \ll_{\Phi_v,K_{f,v}}
        \Delta_v^{1/2}\disc(\nu_v)^{-1/4}
        \leq
        \Delta_v^{1/2}.
\tag{3}
\]
Thus the second estimate is uniform in the ramification of $E_v/F_v$; no upper
bound for $\disc(\nu_v)$ is required, since its exponent in $(2)$ is negative.
The order-saving estimate used in the product below is $(1)$.

Multiplying the preceding local bound over all finite places gives
\[
        \frac{
        \prod_{v\neq\infty}
        \left|
        \mathcal Z_v^\natural(\Phi_{v,E,R},\omega_v)
        \right|
        }{
        \vol(K_{T,R})
        }
        \ll_{\Phi,U}
        \left(\prod_{v\neq\infty}\Delta_v\right)^{1/2-\eta}
        \prod_{v\neq\infty}(n_v(E,R)+1)^{A_0}.
\]
By the realizability condition \eqref{e:realizable}, the order $\Lambda_v$ is
$R_v$ at every finite place.  Hence
\[
        \Delta_v
        =
        \frac{\disc(R_v)}{\disc(\mathcal O_{E_v})}
        =
        q_v^{b_v(E,R)},
        \qquad
        r_v=b_v(E,R),
\]
and therefore
\[
        \prod_{v\neq\infty}\Delta_v
        =
        \prod_{v\neq\infty}q_v^{b_v(E,R)}
        =
        \Delta_R.
\]
By the divisor bound \eqref{eq:divisor-factor-bound}, applied to the effective
divisor $\mathfrak D_{E,R}$, one has, for every $\varepsilon>0$,
\[
        \prod_{v\neq\infty}(n_v(E,R)+1)^{A_0}
        \ll_{F,A_0,\varepsilon}
        D(E,R)^\varepsilon =D_E^\varepsilon\Delta_R^\varepsilon.
\]
Therefore
\[
        \frac{
        \prod_{v\neq\infty}
        \left|
        \mathcal Z_v^\natural(\Phi_{v,E,R},\omega_v)
        \right|
        }{
        \vol(K_{T,R})
        }
        \ll_{\Phi,U,\varepsilon}
        D_E^\varepsilon
        \Delta_R^{1/2-\eta+\varepsilon}.
\]
\end{proof}

\begin{lemma}
\label{lem:eisenstein-weyl-sum-decay}
Assume $D\simeq M_2(F)$.  For every $\varphi\in C_c^\infty(X_U)$,
\[
    W(P_{\mathrm{Eis}}\varphi;E,R)\rightarrow 0
    \qquad
    \text{as}
    \quad
    \deg\mathfrak D_{E,R}\rightarrow\infty .
\]
\end{lemma}

\begin{proof}
Since $\varphi$ is smooth, its lift to $[G]$ is fixed by some compact open
subgroup $J\subset G(\mathbb A_F)$. The Eisenstein subspace is
invariant under the right regular representation. Thus, $P_{\mathrm{Eis}}$ commutes with
right translations, and only $J$-fixed Eisenstein data occur in the spectral
expansion of $P_{\mathrm{Eis}}\varphi$.

It is enough first to treat a normalized Eisenstein series
$\mathcal E(g,\Phi,\omega)$ occurring in the fixed $U$-level unitary Eisenstein
spectrum.  By the packet-period comparison and the bound $w_{E,R}\ll 1$,
\[
        \left|
        W(\mathcal E(\cdot,\Phi,\omega);E,R)
        \right|
        \ll_U
        \frac{
        \left|\mathcal P_{E,R}(\Phi,\omega)\right|
        }{
        \operatorname{vol}(K_T)\,\#\calp_{E,R}
        }.
\]
By the unfolding formula, Lemma~\ref{lem:eis-unfolding},
\[
        \mathcal P_{E,R}(\Phi,\omega)
        =
        \frac{
        L_E\!\left(\frac12,\omega\circ N_{E/F}\right)
        }{
        L_F(1,\omega^2)
        }
        \prod_v
        \mathcal Z_v^\natural(\Phi_{v,E,R},\omega_v).
\]
Since the corresponding principal series has a nonzero $J$-fixed vector, the
conductor of $\omega$ is bounded in terms of $J$.  Hence the function field Riemann hypothesis for abelian $L$-functions \cite{deligne-weilII}
gives, for every $\varepsilon>0$,
\[
        L_E\!\left(\frac12,\omega\circ N_{E/F}\right)
        \ll_{U,\varepsilon}
        D_E^\varepsilon .
\]
The constant residual representation has been separated off, so on the remaining
unitary Eisenstein spectrum
\[
        L_F(1,\omega^2)^{-1}\ll_U 1 .
\]

The factor at $\infty$ and the infinite toric volume are fixed by
$E_\infty\simeq F_{\infty,\star}$ and $U_\infty=T_\star(F_\infty)$, hence are absorbed into
the implied constant.  
Applying Proposition~\ref{prop:eis-local-subconvex} with
$\varepsilon/3$, we obtain
\[
        \frac{
        \prod_v
        \left|
        \mathcal Z_v^\natural(\Phi_{v,E,R},\omega_v)
        \right|
        }{
        \operatorname{vol}(K_T)
        }
        \ll_{\Phi,U,\varepsilon}
        D_E^{\varepsilon/3}
        \Delta_R^{1/2-\eta+\varepsilon/3}.
\]
Combining this with the preceding abelian $L$-value estimate, also applied with
$\varepsilon/3$, gives
\[
        \left|
        W(\mathcal E(\cdot,\Phi,\omega);E,R)
        \right|
        \ll_{\Phi,U,\varepsilon}
        D_E^{2\varepsilon/3}
        \frac{
        \Delta_R^{1/2-\eta+\varepsilon/3}
        }{
        \#\calp_{E,R}
        }.
\]
By Lemma~\ref{lem:packet-size-lower-bound} we have
$\#\calp_{E,R}
        \gg_{F,U,\varepsilon}
        (D_E\Delta_R)^{1/2-\varepsilon/3}$.
Therefore
\[
        \left|
        W(\mathcal E(\cdot,\Phi,\omega);E,R)
        \right|
        \ll_{\Phi,U,\varepsilon}
        D_E^{-1/2+\varepsilon}
        \Delta_R^{-\eta+\varepsilon}.
\]
Choosing $0<\delta<\min\{1/2,\eta\}$ gives
$\left|
        W(\mathcal E(\cdot,\Phi,\omega);E,R)
        \right|
        \ll_{\Phi,U,\varepsilon}
        D(E,R)^{-\delta+\varepsilon}$.

Finally, by  \cite[VI.2.7]{moeglin-waldspurger}, the Eisenstein projection
$P_{\mathrm{Eis}}\varphi$ is represented by an integral over the fixed $U$-level
unitary Eisenstein spectrum.  Since $\varphi$ is fixed, its Eisenstein spectral
coefficients are integrable, and the preceding bound is uniform on this fixed
spectrum.  Integrating gives
$W(P_{\mathrm{Eis}}\varphi;E,R)
        \rightarrow 0$.
\end{proof}

\subsubsection{Proof of the split theorem}
\label{subsubsec:proof-split-theorem}

\begin{proof}[Proof of Theorem~\ref{thm:split-adic-duke}]
The finite-component hypothesis
$\eta_{E_n}\notin\mathfrak X_U-\{1\}$ for $n\gg 0$, by Corollary~\ref{cor:component-character-decay}, that
\[
        W(\xi_\eta;E_n,R_n)\rightarrow 0
\]
for every nontrivial component character $\xi_\eta$.

Let $h\in C_c^\infty(X_U)$.  Its lift to $[G]$ is locally constant and right
$U$-invariant, hence is fixed by some compact open subgroup $J\subset G(\mathbb A_F)$.
Since the cuspidal projection commutes with right translations,
\[
        h_{\mathrm{cusp}}:=P_{\mathrm{cusp}}h
        \in L^2_{\mathrm{cusp}}([G])^J\cap L^2([G])^U .
\]
For $G=\PGL_2$ over a function field, $L^2_{\mathrm{cusp}}([G])^J$ is
finite-dimensional.  Thus $h_{\mathrm{cusp}}$ is a finite linear combination of
$U$-fixed cuspidal automorphic vectors, and Lemma~\ref{l:cuspidal-weyl-sum-decay},
with the trivial packet character $\chi=1$, gives $W(h_{\mathrm{cusp}};E_n,R_n)\rightarrow 0$.
The Eisenstein contribution is handled by
Lemma~\ref{lem:eisenstein-weyl-sum-decay}, namely
$W(h_{\mathrm{Eis}};E_n,R_n)\rightarrow 0$.
Therefore all hypotheses of Proposition~\ref{prop:spectral-reduction-noncompact}
are satisfied.  
\end{proof}

\end{document}